\documentclass[12pt,twoside,american,english,british]{article}
\usepackage[T1]{fontenc}
\usepackage[utf8]{inputenc}
\usepackage[a4paper]{geometry}
\usepackage{color}
\usepackage{amsmath}
\usepackage{amsthm}
\usepackage{amssymb}
\usepackage{esint}
\usepackage{comment}

\makeatletter
\numberwithin{equation}{section}
\numberwithin{figure}{section}
\theoremstyle{plain}
\newtheorem{thm}{\protect\theoremname}[section]
\theoremstyle{plain}
\newtheorem{cor}[thm]{\protect\corollaryname}
\theoremstyle{definition}
\newtheorem{defn}[thm]{\protect\definitionname}
\theoremstyle{plain}
\newtheorem{lem}[thm]{\protect\lemmaname}
\theoremstyle{plain}
\newtheorem{prop}[thm]{\protect\propositionname}
\newcommand{\lyxaddress}[1]{
	\par {\raggedright #1
	\vspace{1.4em}
	\noindent\par}
}

\usepackage[T1]{fontenc}
\usepackage[utf8]{inputenc}
\usepackage[a4paper]{geometry}
\usepackage{color}
\usepackage{dsfont}
\usepackage{amsmath}
\usepackage{amsthm}
\usepackage{amssymb}
\usepackage{setspace}
\usepackage{esint}

\makeatletter
\numberwithin{equation}{section}
\numberwithin{figure}{section}
\theoremstyle{plain}

\@ifundefined{date}{}{\date{}}
\usepackage{babel}
\usepackage{latexsym}
\usepackage{amsthm}
\usepackage{esint}
\usepackage{eucal}
\usepackage{epstopdf}
\usepackage{graphicx}
\usepackage{bigints}
\theoremstyle{plain}

\newtheoremstyle{boldremark}
    {\dimexpr\topsep/2\relax} 
    {\dimexpr\topsep/2\relax} 
    {}          
    {}          
    {\bfseries} 
    {.}         
    {.5em}      
    {}          

\theoremstyle{boldremark}
\newtheorem{brem} [thm] {Remark} 

\usepackage{fancyhdr}
\usepackage{blindtext}
\usepackage{titlesec}

\titleformat{name=\chapter}[display]
{\normalfont\Huge\itshape}
{\titlerule[1pt]\vspace{-33pt}\filleft%
  \parbox[t]{6em}{%
    \raggedleft%
    \rule{\linewidth}{0.5ex}\newline%
    \chaptertitlename\ \thechapter%
  }%
}
{4pc}
{\normalfont\upshape\bfseries\Huge}
\allowdisplaybreaks

\makeatother

\usepackage{babel}
\addto\captionsbritish{\renewcommand{\definitionname}{Definition}}
\addto\captionsbritish{\renewcommand{\lemmaname}{Lemma}}
\addto\captionsbritish{\renewcommand{\theoremname}{Theorem}}
\addto\captionsenglish{\renewcommand{\definitionname}{Definition}}
\addto\captionsenglish{\renewcommand{\lemmaname}{Lemma}}
\addto\captionsenglish{\renewcommand{\theoremname}{Theorem}}
\providecommand{\definitionname}{Definition}
\providecommand{\lemmaname}{Lemma}
\providecommand{\theoremname}{Theorem}

\usepackage[colorlinks,pdfpagelabels,pdfstartview = FitH,bookmarksopen
= true,bookmarksnumbered = true,linkcolor = red,plainpages =
false,hypertexnames = false,citecolor = blue,pagebackref=false,urlcolor=blue]{hyperref}

\makeatother

\usepackage{babel}
\addto\captionsamerican{\renewcommand{\corollaryname}{Corollary}}
\addto\captionsamerican{\renewcommand{\definitionname}{Definition}}
\addto\captionsamerican{\renewcommand{\lemmaname}{Lemma}}
\addto\captionsamerican{\renewcommand{\propositionname}{Proposition}}
\addto\captionsamerican{\renewcommand{\theoremname}{Theorem}}
\addto\captionsbritish{\renewcommand{\corollaryname}{Corollary}}
\addto\captionsbritish{\renewcommand{\definitionname}{Definition}}
\addto\captionsbritish{\renewcommand{\lemmaname}{Lemma}}
\addto\captionsbritish{\renewcommand{\propositionname}{Proposition}}
\addto\captionsbritish{\renewcommand{\theoremname}{Theorem}}
\addto\captionsenglish{\renewcommand{\corollaryname}{Corollary}}
\addto\captionsenglish{\renewcommand{\definitionname}{Definition}}
\addto\captionsenglish{\renewcommand{\lemmaname}{Lemma}}
\addto\captionsenglish{\renewcommand{\propositionname}{Proposition}}
\addto\captionsenglish{\renewcommand{\theoremname}{Theorem}}
\providecommand{\corollaryname}{Corollary}
\providecommand{\definitionname}{Definition}
\providecommand{\lemmaname}{Lemma}
\providecommand{\propositionname}{Proposition}
\providecommand{\theoremname}{Theorem}

\begin{document}
\title{\selectlanguage{english}%
\textbf{On the second-order regularity of solutions to widely singular or degenerate elliptic
equations}}
\author{\selectlanguage{english}%
Pasquale Ambrosio, Antonio Giuseppe Grimaldi, Antonia Passarelli di
Napoli}
\date{\selectlanguage{english}%
May 30, 2025}
\maketitle
\begin{abstract}
\noindent We consider local weak solutions to PDEs of the type 
\[
-\,\mathrm{div}\left((\vert Du\vert-\lambda)_{+}^{p-1}\frac{Du}{\vert Du\vert}\right)=f\,\,\,\,\,\,\,\text{in}\,\,\Omega,
\]
where $1<p<\infty$, $\Omega$ is an open subset of $\mathbb{R}^{n}$ for $n\geq2$, $\lambda$ is a positive constant and $(\,\cdot\,)_{+}$ stands for the positive part. Equations of this form are widely degenerate for $p\ge 2$ and widely singular for $1<p<2$. We establish higher differentiability results for a suitable nonlinear function of the gradient $Du$ of the local weak solutions, assuming that $f$ belongs to the local Besov space $B^{(p-2)/p}_{p',1,loc}(\Omega)$ when $p>2$, and that $f\in L_{loc}^{{\frac{np}{n(p-1)+2-p}}}(\Omega)$ if $1<p\leq2$. The conditions on the datum $f$ are essentially sharp. As a consequence, we obtain the local higher integrability of $Du$ under the same minimal assumptions on $f$. For $\lambda=0$, our results give back those contained in \cite{CGPdN,Irv}.
\end{abstract}


\section{Introduction}

\noindent $\hspace*{1em}$In this paper, we are interested in the local
weak solutions to the strongly singular or degenerate elliptic equation
\begin{equation}
-\,\mathrm{div}\left((\vert Du\vert-\lambda)_{+}^{p-1}\frac{Du}{\vert Du\vert}\right)=f\,\,\,\,\,\mathrm{in}\,\,\Omega,\label{eq:degenerate}
\end{equation}
where $1<p<\infty$, $\Omega$ is an open subset of $\mathbb{R}^{n}$
($n\geq2$), $\lambda>0$ is a fixed parameter and $(\,\cdot\,)_{+}$
stands for the positive part. The peculiarity of equation \eqref{eq:degenerate}
is that it is uniformly elliptic only outside the ball centered at
the origin with radius $\lambda$, where its principal part behaves
asymptotically as the classical $p$-Laplace operator. Therefore,
the study of such an equation fits into the wider class of the asymptotically
regular problems that have been extensively studied starting from
the pioneering paper \cite{CE} (see also \cite{CoFi1,CoFi2,CGPdN1,CGPdN2,CGM,EMM,FFM,FPdNV2,GiaMo,LPdNV,Ray,ScheS}
for extensions to various other settings).\\
$\hspace*{1em}$As our main result, here we establish the local Sobolev
regularity of a nonlinear function of the gradient $Du$ of the weak
solutions to equation (\ref{eq:degenerate}), by assuming that the
datum $f$ belongs to a suitable local Besov space when $p>2$ (see
Theorem \ref{thm:theo1} below) and that $f\in L_{loc}^{{\frac{np}{n(p-1)+2-p}}}(\Omega)$
if $1<p\leq2$ (see Theorem \ref{thm:theo2}). These results, in turn,
imply the local higher integrability of $Du$ under the same hypotheses
on the function $f$ (cf. Corollary \ref{cor:corollario3}).\\
$\hspace*{1em}$Before specifying in detail the assumption on $f$
in the case $p>2$, we wish to discuss some results already available
in the literature. A common aspect of nonlinear elliptic problems
with growth rate $p\geq2$ is that the higher differentiability is
proven for a nonlinear function of the gradient that takes into account
the growth of the structure function of the equation. Indeed, already
for the $p$-Poisson equation (which is obtained from (\ref{eq:degenerate})
by setting $\lambda=0$), the higher differentiability is established
for the function 
\[
V_{p}(Du):=\,\vert Du\vert^{\frac{p-2}{2}}Du\,,
\]
as can be seen in many papers, starting from the pioneering one by Uhlenbeck \cite{Uhlenbeck}. In case of widely degenerate problems, this
phenomenon persists and higher differentiability results hold true
for the function
\begin{equation}
H_{\frac{p}{2}}(Du):=\,(\vert Du\vert-\lambda)_{+}^{p/2}\frac{Du}{\left|Du\right|}\label{eq:HP2}
\end{equation}
(see \cite{Am} and \cite[Theorem 4.2]{BraCaSan}). However, this function does
not provide any information about the second-order regularity of the
solutions in the set where the equation becomes degenerate. Actually,
since every $\lambda$-Lipschitz function is a solution of the homogeneous
elliptic equation 
\[
\mathrm{div}\left((\vert Du\vert-\lambda)_{+}^{p-1}\frac{Du}{\vert Du\vert}\right)=0\,,
\]
no more than Lipschitz regularity can be expected for the solutions:
in this regard, see \cite{BoDuGiPa,Brasco,BraCaSan,Gri,Ru}.\\
\noindent$\hspace*{1em}$In this paper, the nonlinear function of the gradient that gains higher weak differentiability needs to be chosen with two main features: on the one hand, as in the case of the function in \eqref{eq:HP2}, it has to vanish in the region $\{\vert Du\vert\leq\lambda\}$; on the other hand, it has to compensate for the loss of uniform ellipticity of equation  \eqref{eq:degenerate} as $\vert Du\vert \to \lambda^+$. Indeed, defining the function $H_{p-1}:\mathbb{R}^{n}\rightarrow\mathbb{R}^{n}$
by 
\begin{equation*}
H_{p-1}(\xi):=\begin{cases}
\begin{array}{cc}
(\vert\xi\vert-\lambda)_{+}^{p-1}\,\frac{\xi}{\left|\xi\right|} & \,\,\mathrm{if}\,\,\,\xi\neq0,\\
0 & \,\,\mathrm{if}\,\,\,\xi=0,
\end{array}\end{cases}\label{eq:Hfun}
\end{equation*}
and denoting, in accordance with \cite{DeMi,MingRad}, the ellipticity ratio of equation (1.1) by
\begin{equation}
\mathcal{R}(\xi)\,:=\,\frac{\text{the highest eigenvalue of $D_{\xi}H_{p-1}(\xi)$}}{\text{the lowest eigenvalue of $D_{\xi}H_{p-1}(\xi)$}}\,\approx\,\frac{\vert \xi\vert}{\vert \xi\vert -\lambda}\,,\,\,\,\,\,\,\,\,\,\,\,\text{for}\,\,\vert \xi\vert > \lambda\,,\label{eq:ellratio}
\end{equation}
one sees that this quotient clearly blows up as $\vert \xi\vert \to \lambda^+$, unless $\lambda=0$ (cf. Lemma \ref{lem:qform} below).\\
\noindent$\hspace*{1em}$To provide context for the main aim of this paper, we would like to point out the paper \cite{Irv}, where Irving and Koch obtained some differentiability results for relaxed minimizers of vectorial convex functionals with non-standard growth of the type
$$\int_\Omega \left[F(x,Du)- f \cdot u\right] dx.$$
In particular, for the weak solutions $u\in W^{1,p}(\Omega)$ of the $p$-Poisson equation, they proved that
\noindent 
\begin{equation}
V_{p}(Du)\,\in\,W_{loc}^{1,2}(\Omega,\mathbb{R}^{n})\label{eq:Uhlenbeck}
\end{equation}
if the function $f$ belongs to the Besov space  $B_{p',1}^{\frac{p-2}{p}}(\Omega)$, with $p>2$. Their assumption on $f$ is essentially sharp,
in the sense that the above result is false if 
\[
f\in B_{p',1}^{s}(\Omega)\,\,\,\,\,\,\,\,\mathrm{with}\,\,\,\,\,s<(p-2)/p\,.
\]
Indeed, Brasco and Santambrogio \cite[Section 5]{BS}
showed with an explicit example that condition \eqref{eq:Uhlenbeck} may not hold if $f$ belongs to a fractional
Sobolev space $W_{loc}^{\sigma,p'}(\mathbb{R}^{n})$ with {$0<\sigma<(p-2)/p$},
which is continuously embedded into $B_{p',1,loc}^{s}(\mathbb{R}^{n})$
whenever $s\in(0,\sigma)$ (see Lemma \ref{lem:emb} below).
\noindent $\hspace*{1em}$This work is in the spirit of the ones mentioned above. Indeed, our main aim here is to find the assumptions to impose on the datum $f$ \textit{in the scale of local Besov or Lebesgue
spaces} to obtain the $W^{1,2}$-regularity of a nonlinear function of the gradient of weak solutions to the widely degenerate or singular equation \eqref{eq:degenerate}. In order to state our main results, we introduce the function
\begin{equation}
\mathcal{G}_{\lambda}(t):=\int_{0}^{t}\frac{\omega^{\frac{p}{2}+\frac{1}{p-1}}}{(\omega+\lambda)^{1+\frac{1}{p-1}}}\,d\omega\,\,\,\,\,\,\,\,\,\,\mathrm{for}\,\,t\geq0.\label{eq:Gfun}
\end{equation}
Moreover, for $\xi\in\mathbb{R}^{n}$ we define
the following vector-valued function:
\begin{equation}
\mathcal{V}_{\lambda}(\xi):=\begin{cases}
\begin{array}{cc}
\mathcal{G}_{\lambda}((\vert\xi\vert-\lambda)_{+})\,\frac{\xi}{\left|\xi\right|} & \,\,\,\mathrm{if}\,\,\,\vert\xi\vert>\lambda,\\
0 & \,\,\,\mathrm{if}\,\,\,\vert\xi\vert\leq\lambda.
\end{array}\end{cases}\label{eq:Vfun}
\end{equation}
Notice that, for $\lambda=0$, we
have 
\begin{equation}
\mathcal{V}_{0}(\xi)=\,\frac{2}{p}\,V_{p}(\xi):=\,\frac{2}{p}\,\vert\xi\vert^{\frac{p-2}{2}}\xi\,.\label{eq:V0}
\end{equation}
At this point, our first result reads as follows.
\begin{thm}
\noindent \label{thm:theo1} \foreignlanguage{british}{Let $n\geq2$,
$p>2$, $\lambda\geq0$ and $f\in B_{p',1,loc}^{\frac{p-2}{p}}(\Omega)$,
}where $p'=p/(p-1)$ is the conjugate exponent of $p$. Moreover, let $u\in W_{loc}^{1,p}(\Omega)$ be a local weak solution
of equation $\mathrm{(\ref{eq:degenerate})}$. Then 
\[
\mathcal{V}_{\lambda}(Du)\,\in\,W_{loc}^{1,2}(\Omega,\mathbb{R}^{n})\,.
\]
Furthermore, for every pair of concentric balls $B_{r}\subset B_{R}\Subset\Omega$
we have 
\begin{align}
 & \int_{B_{r/4}}\left|D\mathcal{V}_{\lambda}(Du)\right|^{2}dx\nonumber \\
 & \,\,\,\,\,\,\,\leq\left(C+\,\frac{C}{r^{2}}\right)\left[1+\lambda^{p}+\,\Vert Du\Vert_{L^{p}(B_{R})}^{p}\,+\Vert f\Vert_{L^{p'}(B_{R})}^{p'}\right]+\,C\,\Vert f\Vert_{B_{p',1}^{\frac{p-2}{p}}(B_{R})}^{p'}\label{eq:estteo1}
\end{align}
for a positive constant $C$ depending only
on $n$, $p$ and $R$.
\end{thm}

\selectlanguage{british}%
\noindent \begin{brem} Looking at (\ref{eq:V0}), one can easily understand
that Theorem \ref{thm:theo1} extends the result
proved in \cite{Irv} to a class of widely degenerate elliptic equations with standard growth, under a sharp assumption on the order of differentiation
of $f$.
\end{brem}\smallskip{}

\selectlanguage{english}%
\noindent $\hspace*{1em}$The proof of the previous theorem is achieved
combining an \textit{a priori} estimate for the solution of a suitable
approximating problem with a comparison estimate. In establishing
the \textit{a priori} estimate, we first need to identify a suitable
function of the gradient that vanishes in the degeneracy set, for
which the second-order \textit{a priori} estimate holds true. Next,
we need to estimate the right-hand side in terms of the derivatives
of such function, without assuming any Sobolev regularity for the
datum $f$. This is done by virtue of the following implication 
\begin{equation}
|g|^{\frac{p-2}{2}}g\in W_{loc}^{1,2}(\Omega)\,\Rightarrow\,g\in B_{p,\infty,loc}^{\frac{2}{p}}(\Omega),\label{imp2}
\end{equation}
which allows us to use the duality of Besov spaces, provided that one imposes
a suitable Besov regularity on the right-hand side $f$. 
{This approach has been inspired by \cite{BS}, where the authors use for the first time a duality-based inequality in the setting of fractional Sobolev spaces, but limiting themselves to the $p$-Poisson equation.} Finally,
we use a comparison argument to transfer the higher differentiability
of the approximating solutions to the solution of equation \eqref{eq:degenerate}.
\noindent \begin{brem} We do not know whether Theorem \ref{thm:theo1}
is still true when we weaken the regularity of $f$ to $B_{p',q,loc}^{(p-2)/p}(\Omega)$
for some $q>1$, also in the case $\lambda=0$. In this regard we
point out the paper \cite{Eb}, where the authors apparently
prove that, for the $p$-Poisson equation, assertion (\ref{eq:Uhlenbeck})
holds under the weaker assumption $f\in W_{loc}^{(p-2)/p,p'}(\Omega)$.
However, we believe that they made a mistake at the beginning of \cite[page 373]{Eb} in applying \cite[Formula (4)]{Eb}.\end{brem}\smallskip{}

\noindent $\hspace*{1em}$We would like to mention that in \cite{CiaMaz18,CiaMaz19}
the authors proved that the assumption $f\in L^{2}$ is sufficient
to prove the $W^{1,2}$-regularity of $\left|Du\right|^{p-2}Du$,
which is of course a different function of the gradient. At the moment,
we do not know whether the analogous result can be obtained for the
solutions of widely degenerate equations.

\noindent $\hspace*{1em}$Now we turn our attention to the sub-quadratic
case, i.e.\ when $1<p\leq2$. It is well known that, already for
the less degenerate case of the $p$-Poisson equation, the higher
differentiability of the solutions can be achieved without assuming
any differentiability on the datum $f$. This different behaviour
can be easily explained by observing that, if $1<p\leq2$, 
\[
|g|^{\frac{p-2}{2}}g\in W_{loc}^{1,2}(\Omega)\,\Rightarrow\,g\in W_{loc}^{1,p}(\Omega),
\]
which of course does not hold true for $p>2$ (compare with \eqref{imp2}).
Therefore, the right-hand side can be estimated without assuming any
differentiability for $f$ (neither of integer nor of fractional order),
but only a suitable degree of integrability. The sharp assumption
on $f$ in the scale of Lebesgue spaces has been recently found in
\cite{CGPdN}.

\noindent $\hspace*{1em}$Here, we prove that {a result analogous
to \cite[Theorem 1.1]{CGPdN} } holds true also when dealing with
solutions of widely {singular} equations. More precisely, we establish
the following result.
\begin{thm}
\label{thm:theo2} Let $n\geq2$, $1<p\leq2$,
$\lambda\geq0$ and $f\in L_{loc}^{{\frac{np}{n(p-1)+2-p}}}(\Omega)$. Moreover, let $u\in W_{loc}^{1,p}(\Omega)$ be a local weak solution
of equation $\mathrm{(\ref{eq:degenerate})}$. Then 
\[
\mathcal{V}_{\lambda}(Du)\,\in\,W_{loc}^{1,2}(\Omega,\mathbb{R}^{n})\,.
\]
Furthermore, for every pair of concentric balls $B_{r}\subset B_{R}\Subset\Omega$
we have
\begin{align}
 & \int_{B_{r/4}}\left|D\mathcal{V}_{\lambda}(Du)\right|^{2}dx\nonumber \\
 & \,\,\,\,\,\,\,\leq\,\frac{C}{r^{2}}\left[1+\lambda^{p}\,+\,\Vert Du\Vert_{L^{p}(B_{R})}^{p}\,+\Vert f\Vert_{L^{\frac{np}{n(p-1)+2-p}}(B_{R})}^{p'}\right]\,+\,C\,\Vert f\Vert_{L^{\frac{np}{n(p-1)+2-p}}(B_{R})}^{p'}\label{eq:estteo2}
\end{align}
for a positive constant $C$ depending only
on $n$, $p$ and $R$.
\end{thm}

\noindent $\hspace*{1em}$As an easy consequence of the higher differentiability
results in Theorems \ref{thm:theo1} and \ref{thm:theo2}, since the
gradient of the solution is bounded in the region $\{\vert Du\vert\leq\lambda\}$ and
$\mathcal{G}_{\lambda}(t)\approx t^{p/2}$ for {large values
of $t$} (see Lemma \ref{growth} below), we are able to establish
{the following} higher integrability result for the gradient of
local weak solutions of \eqref{eq:degenerate}:
\begin{cor}
\noindent \label{cor:corollario3} Under the assumptions of Theorem
\ref{thm:theo1} or Theorem \ref{thm:theo2}, we have
\[
Du\in L_{loc}^{q}(\Omega,\mathbb{R}^{n}),
\]
where 
\begin{align*}
q=\begin{cases}
\begin{array}{cc}
\text{any value in \ensuremath{[1,\infty)}} & \text{if}\,\,\,n=2,\,\,\,\,\,\,\,\,\,\,\,\\
\frac{np}{n-2} & \text{if}\,\,\,n\ge3.\,\,\,\,\,\,\,\,\,\,
\end{array}\end{cases}
\end{align*}
\smallskip{}
\end{cor}

\noindent $\hspace*{1em}$Before describing the structure of this paper, we observe that, if we interpret the inverse of the ratio in \eqref{eq:ellratio} as a weight, the ellipticity bounds of the matrix $D_{\xi}H_{p-1}(Du)$ can be expressed as
\begin{align*}
\frac{c(p)}{\mathcal{R}(Du)}\,(\vert Du\vert-\lambda)^{p-2}_{+}\,|\zeta|^{2}\le\langle D_{\xi}H_{p-1}(Du)\,\zeta,\zeta\rangle\le\,c(p)\,(|Du|-\lambda)_{+}^{p-2}\,|\zeta|^{2},
\end{align*}
for every $\zeta\in \mathbb{R}^n$ (see Lemma \ref{lem:qform} with $\varepsilon=0$). Known results for solutions to non-uniformly elliptic problems rely on integrability properties of $\mathcal{R}(Du)$, which, however, are not available in the present setting. Our choice of the function 
$\mathcal{V}_\lambda$ implicitly incorporates the aforementioned weight.
\\
\noindent $\hspace*{1em}$The paper is organized as follows. Section
\ref{sec:prelim} is devoted to the preliminaries: after a list of
some classic notations and some essentials estimates, we recall the
basic properties of the difference quotients of Sobolev functions.
Section \ref{sec:Besov spaces} is entirely devoted to the definitions
and properties of Besov spaces that will be useful to prove our results.
In Section \ref{sec:a priori}, we establish some \textit{a priori}
estimates that will be needed to demonstrate Theorems \ref{thm:theo1}
and \ref{thm:theo2}, whose proofs are contained in Sections \ref{sec:theo1}
and \ref{sec:teorema2}, respectively.

\section{Preliminaries \label{sec:prelim}}
\selectlanguage{british}%

\subsection{Notation and essential definitions }

\selectlanguage{english}%
\noindent $\hspace*{1em}$In this paper we shall denote by $C$ or
$c$ a general positive constant that may vary on different occasions.
Relevant dependencies on parameters and special constants will be
suitably emphasized using parentheses or subscripts. \foreignlanguage{british}{The
norm we use on $\mathbb{R}^{k}$, }\foreignlanguage{american}{$k\in\mathbb{N}$}\foreignlanguage{british}{,
will be the standard Euclidean one and it will be denoted by $\left|\,\cdot\,\right|$.
In particular, for the vectors $\xi,\eta\in\mathbb{R}^{k}$, we write
$\langle\xi,\eta\rangle$ for the usual inner product and $\left|\xi\right|:=\langle\xi,\xi\rangle^{\frac{1}{2}}$
for the corresponding Euclidean norm.}\\
$\hspace*{1em}$In what follows, $B_{r}(x_{0})=\left\{ x\in\mathbb{R}^{n}:\left|x-x_{0}\right|<r\right\} $
will denote the $n$-dimensional open ball centered at $x_{0}$ with
radius $r$. We shall sometimes omit the dependence on the center
when all balls occurring in a proof are concentric. Unless otherwise
stated, different balls in the same context will have the same center.\\
\foreignlanguage{british}{$\hspace*{1em}$For further needs, we now
define the auxiliary function $H_{\gamma}:\mathbb{R}^{n}\rightarrow\mathbb{R}^{n}$
by 
\begin{equation}
H_{\gamma}(\xi):=\begin{cases}
\begin{array}{cc}
(\vert\xi\vert-\lambda)_{+}^{\gamma}\,\frac{\xi}{\left|\xi\right|} & \,\,\mathrm{if}\,\,\,\xi\neq0,\\
0 & \,\,\mathrm{if}\,\,\,\xi=0,
\end{array}\end{cases}\label{eq:Hfun}
\end{equation}
}

\selectlanguage{british}%
\noindent where $\lambda\geq0$ and $\gamma>0$ are parameters. We
conclude this first part of the preliminaries by recalling the following
definition.
\begin{defn}
\noindent Let $\lambda\geq0$. A function $u\in W_{loc}^{1,p}(\Omega)$
is a \textit{local weak solution} of equation (\ref{eq:degenerate})
if and only if, for any test function $\varphi\in C_{0}^{\infty}(\Omega)$,
the following integral identity holds:
\[
\int_{\Omega}\langle H_{p-1}(Du),D\varphi\rangle\,dx\,=\,\int_{\Omega}f\varphi\,dx.
\]
\end{defn}

\subsection{Algebraic inequalities }

\noindent $\hspace*{1em}$In this section, we gather some relevant
algebraic inequalities that will be needed later on. The first result
follows from an elementary computation.
\begin{lem}
\noindent \label{lem:alge1} For $\xi,\eta\in\mathbb{R}^{n}\setminus\{0\}$,
we have 
\[
\left|\frac{\xi}{\vert\xi\vert}-\frac{\eta}{\vert\eta\vert}\right|\leq\mathrm{\frac{2}{\vert\eta\vert}\,\vert\xi-\eta\vert.}
\]
\end{lem}

\noindent We now recall the following estimate, whose proof can be
found in \cite[Chapter 12]{Lind}.
\begin{lem}
\noindent \label{lem:Lind} Let $p\in(2,\infty)$ and $k\in\mathbb{N}$.
Then, for every $\xi,\eta\in\mathbb{R}^{k}$ we get 
\[
\vert\xi-\eta\vert^{p}\leq\,C\left|\vert\xi\vert^{\frac{p-2}{2}}\xi-\vert\eta\vert^{\frac{p-2}{2}}\eta\right|^{2}
\]
for a constant $C\equiv C(p)>0$.
\end{lem}

\selectlanguage{english}%
\noindent Combining \cite[Lemma 2.2]{AceFu} with \cite[Formula (2.4)]{GiaMo},
we obtain the following
\begin{lem}
\label{D1} Let $1<p<\infty$. There exists a constant $c\equiv c(n,p)>0$
such that 
\begin{center}
$c^{-1}(|\xi|^{2}+|\eta|^{2})^{\frac{p-2}{2}}\leq\dfrac{\left|\vert\xi\vert^{\frac{p-2}{2}}\xi-\vert\eta\vert^{\frac{p-2}{2}}\eta\right|^{2}}{|\xi-\eta|^{2}}\leq c\,(|\xi|^{2}+|\eta|^{2})^{\frac{p-2}{2}}$ 
\par\end{center}
\noindent for every $\xi,\eta\in\mathbb{R}^{n}$ with $\xi\neq\eta$. 
\end{lem}

\selectlanguage{british}%
\noindent $\hspace*{1em}$\foreignlanguage{english}{For the function
$H_{p-1}$ defined by (\ref{eq:Hfun}) with $\gamma=p-1$, we record
the following estimates, }which can be obtained by suitably modifying
the proofs of\foreignlanguage{english}{ \cite[Lemma 4.1]{BraCaSan}
(for the case $p\ge2$), \cite[Lemma 2.5]{Am} (for the case $1<p<2$)
and} \cite[Lemma 2.8]{BoDuGiPa}. 
\selectlanguage{english}%
\begin{lem}
\label{lem:Brasco} Let \foreignlanguage{british}{$p\in(1,\infty)$
and $\lambda\geq0$}. Then, there exists a constant $c\equiv c(n,p)>0$
such that 
\begin{equation}
\langle H_{p-1}(\xi)-H_{p-1}(\eta),\xi-\eta\rangle\,\geq\,c\,\vert H_{\frac{p}{2}}(\xi)-H_{\frac{p}{2}}(\eta)\vert^{2},\label{eq:BraAmb}
\end{equation}
for every $\xi,\eta\in\mathbb{R}^{n}$. \foreignlanguage{british}{Moreover,
if $\vert\eta\vert>\lambda>0$ we have 
\[
\langle H_{p-1}(\xi)-H_{p-1}(\eta),\xi-\eta\rangle\,\geq\,\frac{\min\,\{1,p-1\}}{2^{p+1}}\,\,\frac{(\vert\eta\vert-\lambda)^{p}}{\vert\eta\vert\,(\vert\xi\vert+\vert\eta\vert)}\,\vert\xi-\eta\vert^{2}.
\]
}
\end{lem}

\noindent $\hspace*{1em}$The next result concerns the function $\mathcal{G}_{\lambda}$
defined by (\ref{eq:Gfun}).
\begin{lem}
\noindent \label{lem:Glemma1} Let $p\in(1,\infty)$ and  $\lambda\geq0$. Then 
\[
\mathcal{G}_{\lambda}(t)\,\le\,\frac{2}{p}\,t^{\frac{p}{2}}\left(\frac{t}{t+\lambda}\right)^{\frac{p}{p-1}}
\]
for every $t>0$.
\end{lem}

\noindent \begin{proof}[\bfseries{Proof}] Since the function 
\[
K(\omega):=\left(\frac{\omega}{\omega+\lambda}\right)^{\frac{p}{p-1}}\,,\,\,\,\,\,\,\,\,\,\,\omega>0,
\]
is non-decreasing, for every $t>0$ we have 
\[
\mathcal{G}_{\lambda}(t)=\int_{0}^{t}K(\omega)\,\omega^{\frac{p}{2}-1}\,d\omega\,\le\,\left(\frac{t}{t+\lambda}\right)^{\frac{p}{p-1}}\int_{0}^{t}\omega^{\frac{p}{2}-1}\,d\omega\,=\,\frac{2}{p}\,t^{\frac{p}{2}}\left(\frac{t}{t+\lambda}\right)^{\frac{p}{p-1}}\,.
\qedhere
\]
\end{proof}

\noindent The next lemma relates the function $\mathcal{V}_{\lambda}(\xi)$
with $H_{p-1}(\xi)$.
\begin{lem}
\label{lem:vital} Let $p\in(1,\infty)$ and $\lambda\geq0$.
Then, there exists a constant $C\equiv C(n,p)>0$ such that 
\begin{equation}
\vert\mathcal{V}_{\lambda}(\xi)-\mathcal{V}_{\lambda}(\eta)\vert^{2}\,\leq\,C\,\langle H_{p-1}(\xi)-H_{p-1}(\eta),\xi-\eta\rangle\label{eq:VH}
\end{equation}
for every $\xi,\eta\in\mathbb{R}^{n}$. 
\end{lem}

\noindent \begin{proof}[\bfseries{Proof}] For $\lambda=0$, estimate
(\ref{eq:VH}) boils down to (\ref{eq:BraAmb}). Therefore, from now
on we shall assume that $\lambda>0$. We first note that inequality
(\ref{eq:VH}) is trivially satisfied when $\vert\xi\vert,\vert\eta\vert\leq\lambda$.
If $\vert\eta\vert\leq\lambda<\vert\xi\vert$, using the definitions
(\ref{eq:Vfun}), (\ref{eq:Gfun}), (\ref{eq:Hfun}) and Lemma \ref{lem:Brasco},
we obtain\begin{align*}
\vert\mathcal{V}_{\lambda}(\xi)-\mathcal{V}_{\lambda}(\eta)\vert^{2}\,&=\,[\mathcal{G}_{\lambda}(\vert\xi\vert-\lambda)]^{2}\,\le\left(\int_{0}^{\vert\xi\vert-\lambda}\omega^{\frac{p}{2}-1}\,d\omega\right)^{2}=\,\frac{4}{p^{2}}\,(\vert\xi\vert-\lambda)^{p}\\
&=\,\frac{4}{p^{2}}\,\vert H_{\frac{p}{2}}(\xi)-H_{\frac{p}{2}}(\eta)\vert^{2}\,\le\,c(n,p)\,\langle H_{p-1}(\xi)-H_{p-1}(\eta),\xi-\eta\rangle.
\end{align*}Now let $\vert\xi\vert,\vert\eta\vert>\lambda$. Without loss of generality,
we may assume that $\vert\eta\vert\geq\vert\xi\vert>\lambda$. This
implies 
\begin{equation}
\vert\eta\vert^{2}=\,\frac{\vert\eta\vert\,(\vert\eta\vert+\vert\eta\vert)}{2}\,\geq\,\frac{\vert\eta\vert\,(\vert\xi\vert+\vert\eta\vert)}{2}\,.\label{eq:implication}
\end{equation}
Moreover, we have\begin{align*}
\vert\mathcal{V}_{\lambda}(\xi)-\mathcal{V}_{\lambda}(\eta)\vert\,&=\,\left|\mathcal{V}_{\lambda}(\xi)-\mathcal{G}_{\lambda}(\vert\eta\vert-\lambda)\,\frac{\xi}{\left|\xi\right|}\,+\mathcal{G}_{\lambda}(\vert\eta\vert-\lambda)\,\frac{\xi}{\left|\xi\right|}\,-\mathcal{V}_{\lambda}(\eta)\right|\\
&\leq\,\left|\mathcal{G}_{\lambda}(\vert\xi\vert-\lambda)-\mathcal{G}_{\lambda}(\vert\eta\vert-\lambda)\right|\,+\,\mathcal{G}_{\lambda}(\vert\eta\vert-\lambda)\left|\frac{\xi}{\left|\xi\right|}-\frac{\eta}{\left|\eta\right|}\right|\\
& \le\,\int_{\vert\xi\vert-\lambda}^{\vert\eta\vert-\lambda}\frac{\omega^{\frac{p}{2}+\frac{1}{p-1}}}{(\omega+\lambda)^{1+\frac{1}{p-1}}}\,d\omega\,+\,\frac{2}{\vert\eta\vert}\,\vert\xi-\eta\vert\int_{0}^{\vert\eta\vert-\lambda}\frac{\omega^{\frac{p}{2}+\frac{1}{p-1}}}{(\omega+\lambda)^{1+\frac{1}{p-1}}}\,d\omega\\
&\le\,\int_{\vert\xi\vert-\lambda}^{\vert\eta\vert-\lambda}\omega^{\frac{p}{2}-1}\,d\omega\,+\,\frac{2}{\vert\eta\vert}\,\vert\xi-\eta\vert\int_{0}^{\vert\eta\vert-\lambda}\omega^{\frac{p}{2}-1}\,d\omega\\
&=\,\frac{2}{p}\,\left|\vert H_{\frac{p}{2}}(\eta)\vert - \vert H_{\frac{p}{2}}(\xi)\vert\right|+\,4\,\frac{(\vert\eta\vert-\lambda)^{\frac{p}{2}}}{p\,\vert\eta\vert}\,\vert\xi-\eta\vert\\
&\le\,\frac{2}{p}\,\vert H_{\frac{p}{2}}(\xi)-H_{\frac{p}{2}}(\eta)\vert\,+\,4\,\frac{(\vert\eta\vert-\lambda)^{\frac{p}{2}}}{p\,\vert\eta\vert}\,\vert\xi-\eta\vert,
\end{align*}where, in the third line, we have used Lemma \ref{lem:alge1} and
the fact that $\mathcal{G}_{\lambda}$ is an increasing function.
Now, applying Young's inequality, estimate (\ref{eq:implication})
and Lemma \ref{lem:Brasco}, we obtain\begin{align*}
\vert\mathcal{V}_{\lambda}(\xi)-\mathcal{V}_{\lambda}(\eta)\vert^{2}\,&\le\,\frac{8}{p^{2}}\,\vert H_{\frac{p}{2}}(\xi)-H_{\frac{p}{2}}(\eta)\vert^{2}\,+\,32\,\frac{(\vert\eta\vert-\lambda)^{p}}{p^{2}\,\vert\eta\vert^{2}}\,\vert\xi-\eta\vert^{2}\\
&\le\,\frac{8}{p^{2}}\,\vert H_{\frac{p}{2}}(\xi)-H_{\frac{p}{2}}(\eta)\vert^{2}\,+\,\frac{64}{p^{2}}\,\frac{(\vert\eta\vert-\lambda)^{p}}{\vert\eta\vert\,(\vert\xi\vert+\vert\eta\vert)}\,\vert\xi-\eta\vert^{2}\\
&\leq\,C(n,p)\,\langle H_{p-1}(\xi)-H_{p-1}(\eta),\xi-\eta\rangle.
\end{align*}This completes the proof.\end{proof}

\noindent $\hspace*{1em}$We conclude this
section with the proof of the following lemma, on which the conclusion
of Corollary \ref{cor:corollario3} is based.
\begin{lem}
\label{growth} Let $p\in(1,\infty)$ and $\lambda>0$.
Then, there exist two positive constants $c\equiv c(p)$ and
$\tilde{c}\equiv\tilde{c}(p)$ such that 
\[
c\,(t+\lambda)^{p/2}-\tilde{c}\,\lambda^{p/2}\le\,\mathcal{G}_{\lambda}(t)\,\le\dfrac{2}{p}\,t^{p/2}
\]
for all $t\ge0$.
\end{lem}

\begin{proof}[\textbf{{Proof}}]
From the very definition of the function $\mathcal{G}_{\lambda}$,
we easily get the upper bound 
\[
\mathcal{G}_{\lambda}(t)\le\int_{0}^{t}\omega^{\frac{p-2}{2}}d\omega=\dfrac{2}{p}\,t^{p/2}\,\,\,\,\,\,\,\,\mathrm{for\,\,all\,\,}t\geq0.
\]
For the derivation of the lower bound, we write the integral that defines $\mathcal{G}_{\lambda}(t)$
 as follows:
\begin{align}\label{est:lem1}
\int_{0}^{t}\frac{\omega^{\frac{p}{2}+\frac{1}{p-1}}}{(\omega+\lambda)^{1+\frac{1}{p-1}}}\,d\omega\,=\int_{0}^{t}\frac{(\omega+\lambda-\lambda)^{\frac{p}{2}+\frac{1}{p-1}}}{(\omega+\lambda)^{1+\frac{1}{p-1}}}\,d\omega\,.
\end{align}
Now we recall that for every $\gamma>0$ it holds 
\begin{equation}
2^{-\gamma}a^{\gamma}-b^{\gamma}\le(a-b)^{\gamma},\quad\forall\,a\ge b\ge0\,.\label{ineq}
\end{equation}
Using inequality \eqref{ineq} with $\gamma=\frac{p}{2}+\frac{1}{p-1}$,
 $a=\omega+\lambda$ and $b=\lambda$, we find that 
\begin{align}\label{est:lem2}
\int_{0}^{t}\frac{(\omega+\lambda-\lambda)^{\frac{p}{2}+\frac{1}{p-1}}}{(\omega+\lambda)^{1+\frac{1}{p-1}}}\,d\omega\,&\ge\,2^{\frac{1}{1-p}-\frac{p}{2}}\int_0^t(\omega+\lambda)^{\frac{p}{2}-1}\,d\omega\,-\,\lambda^{\frac{p}{2}+\frac{1}{p-1}}\int_{0}^{t}\dfrac{1}{(\omega+\lambda)^{1+\frac{1}{p-1}}}\,d\omega\nonumber\\
 & =\,c(p)\,(t+\lambda)^{\frac{p}{2}}\,-\,c(p)\,\lambda^{\frac{p}{2}}\,+(p-1)\,\lambda^{\frac{p}{2}+\frac{1}{p-1}}\left[(\omega+\lambda)^{\frac{1}{1-p}}\right]_0^t\nonumber\\
  &\ge\,c(p)\,(t+\lambda)^{\frac{p}{2}}\,-\,c(p)\,\lambda^{\frac{p}{2}}\,-\,(p-1)\,\lambda^{\frac{p}{2}}\,.
\end{align}
Joining \eqref{est:lem1} and \eqref{est:lem2}, we
obtain the asserted lower bound.
\end{proof}

\subsection{Difference quotients}

\label{subsec:DiffOpe}

\noindent $\hspace*{1em}$We recall here the definition and some elementary
properties of the difference quotients that will be useful in the
following (see, for example, \cite{Giu}). 
\begin{defn}
\noindent For every vector-valued function $F:\mathbb{R}^{n}\rightarrow\mathbb{R}^{k}$
the \textit{finite difference operator }in the direction $x_{j}$
is defined by 
\[
\tau_{j,h}F(x)=F(x+he_{j})-F(x),
\]
where $h\in\mathbb{R}$, $e_{j}$ is the unit vector in the direction
$x_{j}$ and $j\in\{1,\ldots,n\}$.\\
 $\hspace*{1em}$The \textit{difference quotient} of $F$ with respect
to $x_{j}$ is defined for $h\in\mathbb{R}\setminus\{0\}$ by 
\[
\Delta_{j,h}F(x)\,=\,\frac{\tau_{j,h}F(x)}{h}\,.
\]
\end{defn}

\noindent When no confusion can arise, we shall omit the index $j$
and simply write $\tau_{h}$ or $\Delta_{h}$ instead of $\tau_{j,h}$
or $\Delta_{j,h}$, respectively. 
\selectlanguage{british}%
\begin{prop}
\noindent Let $\Omega\subset\mathbb{R}^{n}$ be an open set and let
$F\in W^{1,q}(\Omega)$, with $q\geq1$. Moreover, let $G:\Omega\rightarrow\mathbb{R}$
be a measurable function and consider the set
\[
\Omega_{\vert h\vert}:=\left\{ x\in\Omega:\mathrm{dist}\left(x,\partial\Omega\right)>\vert h\vert\right\} .
\]
\foreignlanguage{english}{Then:}\\
\foreignlanguage{english}{}\\
\foreignlanguage{english}{$\mathrm{(}\mathrm{i}\mathrm{)}$ $\Delta_{h}F\in W^{1,q}\left(\Omega_{\vert h\vert}\right)$
and $\partial_{i}(\Delta_{h}F)=\Delta_{h}(\partial_{i}F)$ for every
$\,i\in\{1,\ldots,n\}$.}\\

\selectlanguage{english}%
\noindent $\mathrm{(}\mathrm{ii}\mathrm{)}$ If at least one of the
functions $F$ or $G$ has support contained in $\Omega_{\vert h\vert}$,
then 
\[
\int_{\Omega}F\,\Delta_{h}G\,dx\,=\,-\int_{\Omega}G\,\Delta_{-h}F\,dx.
\]
$\mathrm{(}\mathrm{iii}\mathrm{)}$ We have 
\[
\Delta_{h}(FG)(x)=F(x+he_{j})\Delta_{h}G(x)\,+\,G(x)\Delta_{h}F(x).
\]
\end{prop}

\selectlanguage{english}%
\noindent The next result about the finite difference operator is
a kind of integral version of the Lagrange Theorem and its proof can
be found in \cite[Lemma 8.1]{Giu}. 
\begin{lem}
\noindent \label{lem:Giusti1} If $0<\rho<R$, $\vert h\vert<\frac{R-\rho}{2}$,
$1<q<+\infty$ and $F\in L^{q}(B_{R},\mathbb{R}^{k})$ is such that
$DF\in L^{q}(B_{R},\mathbb{R}^{k\times n})$, then 
\[
\int_{B_{\rho}}\left|\tau_{h}F(x)\right|^{q}dx\,\leq\,c^{q}(n)\,\vert h\vert^{q}\int_{B_{R}}\left|DF(x)\right|^{q}dx.
\]
Moreover 
\[
\int_{B_{\rho}}\left|F(x+he_{j})\right|^{q}dx\,\leq\,\int_{B_{R}}\left|F(x)\right|^{q}dx.
\]
\end{lem}

\noindent Finally, we recall the following fundamental result, whose
proof can be found in \cite[Lemma 8.2]{Giu}: 
\begin{lem}
\noindent \label{lem:RappIncre} Let $F:\mathbb{R}^{n}\rightarrow\mathbb{R}^{k}$,
$F\in L^{q}(B_{R},\mathbb{R}^{k})$ with $1<q<+\infty$. Suppose that
there exist $\rho\in(0,R)$ and a constant $M>0$ such that 
\[
\sum_{j=1}^{n}\int_{B_{\rho}}\left|\tau_{j,h}F(x)\right|^{q}dx\,\leq\,M^{q}\,\vert h\vert^{q}
\]
for every $h\in\mathbb{R}$ with $\vert h\vert<\frac{R-\rho}{2}$.
Then $F\in W^{1,q}(B_{\rho},\mathbb{R}^{k})$. Moreover 
\[
\Vert DF\Vert_{L^{q}(B_{\rho})}\leq M
\]
and 
\[
\Delta_{j,h}F\rightarrow\partial_{j}F\,\,\,\,\,\,\,\,\,\,in\,\,L_{loc}^{q}(B_{R},\mathbb{R}^{k})\,\,\,\,\mathit{as}\,\,h\rightarrow0,
\]
for each $j\in\{1,\ldots,n\}$. 
\end{lem}

\section{Besov spaces \label{sec:Besov spaces}}

\noindent $\hspace*{1em}$Here we recall some essential facts on the
Besov spaces involved in this paper (see, for example, \cite{Triebel}
and \cite{TriebelIV}).\\
 We denote by $\mathcal{S}(\mathbb{R}^{n})$ and $\mathcal{S}'(\mathbb{R}^{n})$
the Schwartz space and the space of tempered distributions on $\mathbb{R}^{n}$,
respectively. If $v\in\mathcal{S}(\mathbb{R}^{n})$, then 
\begin{equation}
\hat{v}(\xi)=(\mathcal{F}v)(\xi)=(2\pi)^{-n/2}\int_{\mathbb{R}^{n}}e^{-i\,\langle x,\xi\rangle}\,v(x)\,dx,\,\,\,\,\,\,\,\,\xi\in\mathbb{R}^{n},\label{eq:Fourier}
\end{equation}
denotes the Fourier transform of $v$. As usual, $\mathcal{F}^{-1}v$
and $v^{\vee}$ stand for the inverse Fourier transform, given by
the right-hand side of (\ref{eq:Fourier}) with $i$ in place of $-i$.
Both $\mathcal{F}$ and $\mathcal{F}^{-1}$ are extended to $\mathcal{S}'(\mathbb{R}^{n})$
in the standard way.\\
 Now, let $\Gamma(\mathbb{R}^{n})$ be the collection of all sequences
$\varphi=\{\varphi_{j}\}_{j=0}^{\infty}\subset\mathcal{S}(\mathbb{R}^{n})$
such that 
\[
\begin{cases}
\begin{array}{cc}
\mathrm{supp}\,\varphi_{0}\subset\{x\in\mathbb{R}^{n}:|x|\le2\}\quad\quad\quad\quad\,\,\,\\
\mathrm{supp}\,\varphi_{j}\subset\{x\in\mathbb{R}^{n}:2^{j-1}\le|x|\le2^{j+1}\} & \mathrm{if}\,\,j\in\mathbb{N},
\end{array}\end{cases}
\]
for every multi-index $\beta$ there exists a positive number $c_{\beta}$
such that 
\[
2^{j\vert\beta\vert}\,\vert D^{\beta}\varphi_{j}(x)\vert\le c_{\beta}\,,\,\,\,\,\,\,\,\,\forall\,j\in\mathbb{N}_{0},\ \forall\,x\in\mathbb{R}^{n}
\]
and 
\[
\sum_{j=0}^{\infty}\varphi_{j}(x)=1\,,\,\,\,\,\,\,\,\,\forall\,x\in\mathbb{R}^{n}.
\]
Then, it is well known that $\Gamma(\mathbb{R}^{n})$ is not empty
(see \cite[Section 2.3.1, Remark 1]{Triebel}). Moreover, if $\{\varphi_{j}\}_{j=0}^{\infty}\in\Gamma(\mathbb{R}^{n})$,
the entire analytic functions $(\varphi_{j}\,\hat{v})^{\vee}(x)$
make sense pointwise in $\mathbb{R}^{n}$ for any $v\in\mathcal{S}'(\mathbb{R}^{n})$.
Therefore, the following definition makes sense: 
\begin{defn}
\noindent Let $s\in\mathbb{R}$, $1\leq p,q\le\infty$ and $\varphi=\{\varphi_{j}\}_{j=0}^{\infty}\in\Gamma(\mathbb{R}^{n})$.
We define the \textit{Besov space} $B_{p,q}^{s}(\mathbb{R}^{n})$
as the set of all $v\in\mathcal{S}'(\mathbb{R}^{n})$ such that 
\begin{equation}
\Vert v\Vert_{B_{p,q}^{s}(\mathbb{R}^{n})}:=\left(\sum_{j=0}^{\infty}2^{jsq}\,\Vert(\varphi_{j}\,\hat{v})^{\vee}\Vert_{L^{p}(\mathbb{R}^{n})}^{q}\right)^{\frac{1}{q}}<+\infty\,\,\,\,\,\,\,\,\,\,\mathrm{if}\,\,q<\infty,\label{eq:quasi-norm}
\end{equation}
and 
\begin{equation}
\Vert v\Vert_{B_{p,q}^{s}(\mathbb{R}^{n})}:=\,\sup_{j\,\in\,\mathbb{N}_{0}}\,2^{js}\,\Vert(\varphi_{j}\,\hat{v})^{\vee}\Vert_{L^{p}(\mathbb{R}^{n})}<+\infty\,\,\,\,\,\,\,\,\,\,\mathrm{if}\,\,q=\infty.\label{eq:quasi-norm2}
\end{equation}
\begin{brem} The space $B_{p,q}^{s}(\mathbb{R}^{n})$ defined above
is a Banach space with respect to the norm $\Vert\cdot\Vert_{B_{p,q}^{s}(\mathbb{R}^{n})}$.
Obviously, this norm depends on the chosen sequence $\varphi\in\Gamma(\mathbb{R}^{n})$,
but this is not the case for the spaces $B_{p,q}^{s}(\mathbb{R}^{n})$
themselves, in the sense that two different choices for the sequence
$\varphi$ give rise to equivalent norms (see \cite[Sections 2.3.2 and 2.3.3]{Triebel}).
This justifies our omission of the dependence on $\varphi$ in the
left-hand side of (\ref{eq:quasi-norm})$-$(\ref{eq:quasi-norm2})
and in the sequel.\end{brem}
\end{defn}

\noindent $\hspace*{1em}$The norms of the \textit{classical Besov
spaces} $B_{p,q}^{s}(\mathbb{R}^{n})$ with $s\in(0,1)$, $1\leq p<\infty$
and $1\le q\le\infty$ can be characterized via differences of the
functions involved, cf. \cite[Section 2.5.12, Theorem 1]{Triebel}.
More precisely, for $h\in\mathbb{R}^{n}$ and a measurable function
$v:\mathbb{R}^{n}\rightarrow\mathbb{R}^{k}$, let us define 
\[
\delta_{h}v(x):=\,v(x+h)-v(x).
\]
Then we have the equivalence
\begin{equation}
\Vert v\Vert_{B_{p,q}^{s}(\mathbb{R}^{n})}\,\approx\,\Vert v\Vert_{L^{p}(\mathbb{R}^{n})}\,+\,[v]_{B_{p,q}^{s}(\mathbb{R}^{n})}\,,\label{eq:equivalence}
\end{equation}

\noindent where 
\begin{equation}
[v]_{B_{p,q}^{s}(\mathbb{R}^{n})}:=\biggl({\displaystyle \int_{\mathbb{R}^{n}}\biggl({\displaystyle \int_{\mathbb{R}^{n}}\dfrac{|\delta_{h}v(x)|^{p}}{|h|^{sp}}\,dx\biggr)^{\frac{q}{p}}\dfrac{dh}{|h|^{n}}\biggr)^{\frac{1}{q}},\,\,\,\,\,\,\,\,\,\,\text{if}\,\,\,1\le q<\infty},}\label{eq:BeNorm1}
\end{equation}

\noindent and 
\begin{equation}
[v]_{B_{p,\infty}^{s}(\mathbb{R}^{n})}:={\displaystyle \sup_{h\,\in\,\mathbb{R}^{n}}\biggl({\displaystyle \int_{\mathbb{R}^{n}}\dfrac{|\delta_{h}v(x)|^{p}}{|h|^{sp}}\,dx\biggr)^{\frac{1}{p}}}}.\label{eq:BeNorm2}
\end{equation}

\noindent In (\ref{eq:BeNorm1}), if one simply integrates for $\vert h\vert<r$
for a fixed $r>0$, then an equivalent norm is obtained, since 
\begin{center}
$\biggl({\displaystyle \int_{\{|h|\,\geq\,r\}}\biggl({\displaystyle \int_{\mathbb{R}^{n}}\dfrac{|\delta_{h}v(x)|^{p}}{|h|^{sp}}\,dx\biggr)^{\frac{q}{p}}\dfrac{dh}{|h|^{n}}\biggr)^{\frac{1}{q}}\leq\,c(n,s,p,q,r)\,\Vert v\Vert_{L^{p}(\mathbb{R}^{n})}}}\,.$ 
\par\end{center}

\noindent Similarly, in (\ref{eq:BeNorm2}) one can simply take the
supremum over $|h|\leq r$ and obtain an equivalent norm. By construction,
$B_{p,q}^{s}(\mathbb{R}^{n})\subset L^{p}(\mathbb{R}^{n})$.\\
$\hspace*{1em}${Let $\Omega$ be an arbitrary open
set in $\mathbb{R}^{n}$. As usual, $\mathcal{D}(\Omega)=C_{0}^{\infty}(\Omega)$
stands for the space of all infinitely differentiable functions in
$\mathbb{R}^{n}$ with compact support in $\Omega$. Let $\mathcal{D}'(\Omega)$
be the dual space of all distributions in $\Omega$ and let $g\in\mathcal{S}'(\mathbb{R}^{n})$.
Then we denote by $g\vert_{\Omega}$ its restriction to $\Omega$,
i.e. 
\[
g\vert_{\Omega}\in\mathcal{D}'(\Omega):\,\,\,\,\,(g\vert_{\Omega})(\phi)=g(\phi)\,\,\,\,\,\,\,\mathrm{for}\,\,\phi\in\mathcal{D}(\Omega).
\]
\begin{defn}
\noindent Let $\Omega$ be an arbitrary domain in $\mathbb{R}^{n}$
with $\Omega\neq\mathbb{R}^{n}$ and let $s\in\mathbb{R}$, $1\leq p\le\infty$
and $1\leq q\le\infty$. Then 
\[
B_{p,q}^{s}(\Omega):=\left\{ v\in\mathcal{D}'(\Omega):\,v=g\vert_{\Omega}\,\,\,\mathrm{for\,\,some}\,\,g\in B_{p,q}^{s}(\mathbb{R}^{n})\right\} 
\]
and
\[
\Vert v\Vert_{B_{p,q}^{s}(\Omega)}:=\,\inf\,\Vert g\Vert_{B_{p,q}^{s}(\mathbb{R}^{n})}\,,
\]
where the infimum is taken over all $g\in B_{p,q}^{s}(\mathbb{R}^{n})$
such that $g\vert_{\Omega}=v$.
\end{defn}

\noindent $\hspace*{1em}$If $\Omega$ is a bounded $C^{\infty}$-domain
in $\mathbb{R}^{n}$, then the \textit{restriction operator} 
\[
\mathrm{re}_{\Omega}:\mathcal{S}'(\mathbb{R}^{n})\hookrightarrow\mathcal{D}'(\Omega),\,\,\,\,\,\,\mathrm{re}_{\Omega}(v)=v\vert_{\Omega}
\]
generates a linear and bounded map from $B_{p,q}^{s}(\mathbb{R}^{n})$
onto $B_{p,q}^{s}(\Omega)$. Furthermore, the spaces $B_{p,q}^{s}(\Omega)$
satisfy the so-called \textit{extension property}, as ensured by the
next theorem. 
\begin{thm}
\noindent \label{thm:extension} Let $s\in\mathbb{R}$, let $1\leq p,q\le\infty$
and let $\Omega$ be a bounded $C^{\infty}$-domain in $\mathbb{R}^{n}$.
Then, there exists a linear and bounded extension operator $\mathrm{ext}_{\Omega}:B_{p,q}^{s}(\Omega)\hookrightarrow B_{p,q}^{s}(\mathbb{R}^{n})$
such that $\mathrm{re}_{\Omega}\circ\mathrm{ext}_{\Omega}=\mathrm{id}$,
where $\mathrm{id}$ is the identity in $B_{p,q}^{s}(\Omega)$. 
\end{thm}

\noindent We refer to \cite[Theorem 2.82]{TriebelIV} and \cite[Theorem 3.3.4]{Triebel} for a proof
of the previous theorem.

\noindent $\hspace*{1em}$If $s\in\mathbb{R}$, $1\leq p<\infty$
and $1\leq q<\infty$, then $\mathcal{S}(\mathbb{R}^{n})$ is a dense
subset of $B_{p,q}^{s}(\mathbb{R}^{n})$ (cf. \cite[Theorem 2.3.3]{Triebel}).
Consequently, in that case, a continuous linear functional on $B_{p,q}^{s}(\mathbb{R}^{n})$
can be interpreted in the usual way as an element of $\mathcal{S}'(\mathbb{R}^{n})$.
More precisely, $g\in\mathcal{S}'(\mathbb{R}^{n})$ belongs to the
dual space $(B_{p,q}^{s}(\mathbb{R}^{n}))'$ of the space $B_{p,q}^{s}(\mathbb{R}^{n})$
if and only if there exists a positive number $c$ such that
\[
\vert g(\phi)\vert\leq\,c\,\Vert\phi\Vert_{B_{p,q}^{s}(\mathbb{R}^{n})}\,\,\,\,\,\,\,\,\,\,\mathrm{for\,\,all\,\,}\phi\in\mathcal{S}(\mathbb{R}^{n})\,.
\]
We endow $(B_{p,q}^{s}(\mathbb{R}^{n}))'$ with the natural dual norm,
defined by 
\[
\Vert g\Vert_{(B_{p,q}^{s}(\mathbb{R}^{n}))'}=\,\sup\,\left\{ \vert g(\phi)\vert:\phi\in\mathcal{S}(\mathbb{R}^{n})\,\,\,\mathrm{and}\,\,\,\Vert\phi\Vert_{B_{p,q}^{s}(\mathbb{R}^{n})}\leq1\right\} ,\,\,\,\,\,\,\,\,g\in(B_{p,q}^{s}(\mathbb{R}^{n}))'.
\]
Now we recall the following duality formula, which has to be meant
as an isomorphism of normed spaces (see \cite[Theorem 2.11.2]{Triebel}). 
\begin{thm}
\noindent \label{thm:duality00} Let $s\in\mathbb{R}$, $1\le p<\infty$
and $1\le q<\infty$. Then 
\[
(B_{p,q}^{s}(\mathbb{R}^{n}))'=B_{p',q'}^{-s}(\mathbb{R}^{n})\,,
\]
where $p'=\infty$ if $p=1$ (similarly for $q'$).
\end{thm}

\noindent \begin{brem} The restrictions $p<\infty$ and $q<\infty$
in Theorem \ref{thm:duality00} are natural, since, if either $p=\infty$
or $q=\infty$, then $\mathcal{S}(\mathbb{R}^{n})$ is not dense in
$B_{p,q}^{s}(\mathbb{R}^{n})$, and the density of $\mathcal{S}(\mathbb{R}^{n})$
in $B_{p,q}^{s}(\mathbb{R}^{n})$ is the basis of our interpretation
of the dual space $(B_{p,q}^{s}(\mathbb{R}^{n}))'$.\end{brem}\smallskip{}

\noindent $\hspace*{1em}$For our purposes, we now give the following
definition.
\begin{defn}
\noindent For $s\in\mathbb{R}$, $1\le p\leq\infty$ and $1\le q\leq\infty$,
we define $\mathring{B}_{p,q}^{s}(\mathbb{R}^{n})$ as the completion
of $\mathcal{S}(\mathbb{R}^{n})$ in $B_{p,q}^{s}(\mathbb{R}^{n})$
with respect to the norm 
\[
v\mapsto\Vert v\Vert_{B_{p,q}^{s}(\mathbb{R}^{n})}\,.
\]
Of course, only the limit cases $\max\,\{p,q\}=\infty$ are of interest.
We shall denote by $(\mathring{B}_{p,q}^{s}(\mathbb{R}^{n}))'$ the
topological dual of\textit{ }$\mathring{B}_{p,q}^{s}(\mathbb{R}^{n})$,
which is endowed with the natural dual norm 
\[
\Vert g\Vert_{(\mathring{B}_{p,q}^{s}(\mathbb{R}^{n}))'}=\,\sup\,\left\{ \vert g(\phi)\vert:\phi\in\mathcal{S}(\mathbb{R}^{n})\,\,\,\mathrm{and}\,\,\,\Vert\phi\Vert_{B_{p,q}^{s}(\mathbb{R}^{n})}\leq1\right\} ,\,\,\,\,\,\,\,\,g\in(\mathring{B}_{p,q}^{s}(\mathbb{R}^{n}))'.
\]
\end{defn}

\noindent $\hspace*{1em}$The following duality result can be found
in \cite[Section 2.11.2, Remark 2]{Triebel} (see also \cite[pages 121 and 122]{Tri0}).
\begin{thm}
\label{duality} Let $s\in\mathbb{R}$, $1\le p\le\infty$ and $1\le q\le\infty$.
Then 
\[
(\mathring{B}_{p,q}^{s}(\mathbb{R}^{n}))'=B_{p',q'}^{-s}(\mathbb{R}^{n})\,,
\]
where $p'=1$ if $p=\infty$ (similarly for $q'$).
\end{thm}

\noindent $\hspace*{1em}$The next result is a key ingredient for
the proof of Theorem \ref{thm:theo1} and its proof can be found in
\cite[Section 3.3.5]{Triebel}.
\begin{thm}
\label{negder} Let $s\in\mathbb{R}$ and $1\leq p,q\le\infty$. Moreover,
assume that $\Omega$ is a bounded $C^{\infty}$-domain in $\mathbb{R}^{n}$.
Then, for every $v\in B_{p,q}^{s}(\Omega)$ and every $j\in\{1,\ldots,n\}$
we have 
\[
\Vert\partial_{j}v\Vert_{B_{p,q}^{s-1}(\Omega)}\,\le c\,\Vert v\Vert_{B_{p,q}^{s}(\Omega)}
\]
for a positive constant $c$ which is independent of $v$. 
\end{thm}

\noindent $\hspace*{1em}$We can also define local Besov spaces as
follows. Given a domain $\Omega\subset\mathbb{R}^{n}$, we say that
a function $v$ belongs to $B_{p,q,loc}^{s}(\Omega)$ if $\phi\,v\in B_{p,q}^{s}(\mathbb{R}^{n})$
whenever $\phi\in C_{0}^{\infty}(\Omega)$. 
\begin{defn}
\noindent Let $\Omega\subseteq\mathbb{R}^{n}$ be an open set. For
any $s\in(0,1)$ and for any $q\in[1,+\infty)$, we define the \textit{fractional
Sobolev space} $W^{s,q}(\Omega,\mathbb{R}^{k})$ as follows: 
\[
W^{s,q}(\Omega,\mathbb{R}^{k}):=\left\{ v\in L^{q}(\Omega,\mathbb{R}^{k}):\frac{\left|v(x)-v(y)\right|}{\left|x-y\right|^{\frac{n}{q}\,+\,s}}\,\in L^{q}\left(\Omega\times\Omega\right)\right\} ,
\]
i.e. an intermerdiate Banach space between $L^{q}(\Omega,\mathbb{R}^{k})$
and $W^{1,q}(\Omega,\mathbb{R}^{k})$, endowed with the norm 
\[
\Vert v\Vert_{W^{s,q}(\Omega)}:=\,\Vert v\Vert_{L^{q}(\Omega)}\,+\,[v]_{W^{s,q}(\Omega)}\,,
\]
where the term 
\begin{equation}
[v]_{W^{s,q}(\Omega)}:=\left(\int_{\Omega}\int_{\Omega}\frac{\left|v(x)-v(y)\right|^{q}}{\left|x-y\right|^{n\,+\,sq}}\,dx\,dy\right)^{\frac{1}{q}}\label{eq:Gagliardo}
\end{equation}
is the so-called \textit{Gagliardo seminorm} of $v$.

\noindent \begin{brem} For every $s\in(0,1)$ and every $q\in[1,\infty)$,
we have $B_{q,q}^{s}(\mathbb{R}^{n})=W^{s,q}(\mathbb{R}^{n})$. In
fact, using the change of variable $y=x+h$ in (\ref{eq:Gagliardo})
with $\Omega=\mathbb{R}^{n}$, one gets the seminorm (\ref{eq:BeNorm1})
with $p=q$.\end{brem}
\end{defn}

\noindent $\hspace*{1em}$We conclude this section with the following
embedding result, whose proof can be obtained by combining \cite[Section 2.2.2, Remark 3]{Triebel}
with \cite[Section 2.3.2, Proposition 2(ii)]{Triebel}. 
\begin{lem}
\label{lem:emb} Let $s\in(0,1)$ and $q\ge1$. Then, for every $\sigma\in(0,1-s)$
we have the continuous embedding $W_{loc}^{s+\sigma,q}(\mathbb{R}^{n})\hookrightarrow B_{q,1,loc}^{s}(\mathbb{R}^{n})$. 
\end{lem}

\section{Estimates for a regularized problem \label{sec:a priori}}

\noindent $\hspace*{1em}$The aim of this section is to establish
some uniform estimates for the gradient of the weak solutions of a
family of suitable approximating problems. More precisely, let $\lambda\geq0$
and let $u\in W_{loc}^{1,p}(\Omega)$ be a local weak solution of
\eqref{eq:degenerate}, for some $p>1$. Fix an open ball $B_{R}\Subset\Omega$
and assume without loss of generality that $R\le1$. For $\varepsilon\in(0,1]$,
we consider the problem 
\begin{equation}
\begin{cases}
\begin{array}{cc}
-\,\mathrm{div}\,(DG_{\varepsilon}(Du_{\varepsilon}))=f_{\varepsilon} & \mathrm{in}\,\,\,B_{R},\\
u_{\varepsilon}=u & \,\,\,\,\,\,\,\mathrm{on}\,\,\,\partial B_{R},\,\,
\end{array}\end{cases}\label{eq:approximation}
\end{equation}
where:\vspace{0.1cm}

\noindent $\hspace*{1em}\bullet$ $\,\,G_{\varepsilon}(z):=\,\frac{1}{p}\,(\vert z\vert-\lambda)_{+}^{p}\,+\,\frac{\varepsilon}{p}\,(1+\vert z\vert^{2})^{\frac{p}{2}},$
for every $z\in\mathbb{R}^{n}$;\vspace{0.2cm}

\noindent $\hspace*{1em}\bullet$ $\,\,f_{\varepsilon}:=f\ast\phi_{\varepsilon}$
and $\{\phi_{\varepsilon}\}_{\varepsilon\,>\,0}$ is a family of standard
compactly supported $C^{\infty}$ mollifiers.\\

\noindent Observe that 
\begin{equation}
D_{z}G_{\varepsilon}(z)=H_{p-1}(z)\,+\,\varepsilon\,(1+\vert z\vert^{2})^{\frac{p-2}{2}}\,z\,.\label{eq:DG}
\end{equation}
Now we set, for $s>\lambda$, 
\begin{equation}
\boldsymbol{\lambda}(s):=\begin{cases}
\begin{array}{cc}
\dfrac{(s-\lambda)^{p-1}}{s} & \mathrm{if}\,\,p>2,\,\,\,\,\,\,\,\,\,\,\,\\
(p-1)\,\dfrac{(s-\lambda)^{p-1}}{s} & \,\,\mathrm{if}\,\,1<p\le2,\,\,
\end{array}\end{cases}\label{eq:minaut}
\end{equation}
and 
\begin{equation}
\boldsymbol{\Lambda}(s):=\begin{cases}
\begin{array}{cc}
(p-1)\,(s-\lambda)^{p-2} & \mathrm{if}\,\,p>2,\,\,\,\,\,\,\,\,\,\,\,\\
(s-\lambda)^{p-2} & \mathrm{if}\,\,1<p\le2,
\end{array}\end{cases}\label{eq:maxaut}
\end{equation}
and $\boldsymbol{\lambda}(s)=0=\boldsymbol{\Lambda}(s)$ for $0\le s\le\lambda$.
These definitions prove to be useful in the formulation of the next
lemma, whose proof follows from \cite[Lemma 2.7]{BoDuGiPa} (see also
\cite{BoDuGiPa1}) together with standard estimates for the $p$-Laplace
operator. 
\begin{lem}
\label{lem:qform} Let $\varepsilon\in[0,1]$ and $z\in\mathbb{R}^{n}\setminus\{0\}$.
Then, for every $\zeta\in\mathbb{R}^{n}$ we have 
\begin{align*}
\big[\varepsilon\,c_{0}\,(1+|z|^{2})^{\frac{p-2}{2}}+\boldsymbol{\lambda}(|z|)\big]\,|\zeta|^{2}\le\langle D^{2}G_{\varepsilon}(z)\,\zeta,\zeta\rangle\le\big[\varepsilon\,c_{1}\,(1+|z|^{2})^{\frac{p-2}{2}}+\boldsymbol{\Lambda}(|z|)\big]\,|\zeta|^{2},
\end{align*}
where $c_{0}=\min\,\{1,p-1\}$ and $c_{1}=\max\,\{1,p-1\}$. 
\end{lem}

\begin{proof}[\textbf{Proof.}]
Actually, setting for $t>0$ 
\[
h(t)=\frac{(t-\lambda)_{+}^{p-1}}{t}\,\,\,\,\,\,\,\,\,\,\text{and}\,\,\,\,\,\,\,\,\,\,a(t)=(1+t^{2})^{\frac{p-2}{2}},
\]
we can write 
\[
D_{z}G_{\varepsilon}(z)=\Big[h(|z|)+\varepsilon a(|z|)\Big]\,z
\]
and we can easily calculate 
\[
D^{2}G_{\varepsilon}(z)=\Big[h(|z|)+\varepsilon a(|z|)\Big]\mathbb{I}\,+\big[h'(|z|)+\varepsilon a'(|z|)\Big]\frac{z\otimes z}{|z|}\,.
\]
Thus we get 
\[
\langle D^{2}G_{\varepsilon}(z)\,\eta,\zeta\rangle=\Big[h(|z|)+\varepsilon a(|z|)\Big]\langle\eta,\zeta\rangle+\Big[h'(|z|)+\varepsilon a'(|z|)\Big]\sum_{i,j=1}^{n}\frac{z_{i}\,\eta_{i}\,z_{j}\,\zeta_{j}}{|z|}\,,
\]
for any $\eta,\zeta\in\mathbb{R}^{n}$. By the Cauchy-Schwarz inequality,
we have 
\begin{equation}
0\le\sum_{i,j=1}^{n}\frac{z_{i}\,\zeta_{i}\,z_{j}\,\zeta_{j}}{|z|}\le|z|\,|\zeta|^{2}.\label{cs}
\end{equation}
At this point, if {$h'(\vert z\vert)+\varepsilon a'(\vert z\vert)\ge0$}
(which occurs when $p\ge2$), from the lower bound in \eqref{cs}
we immediately obtain 
\[
\langle D^{2}G_{\varepsilon}(z)\,\zeta,\zeta\rangle\ge\big[h(|z|)+\varepsilon a(|z|)\big]|\zeta|^{2}=\,\frac{(|z|-\lambda)_{+}^{p-1}}{|z|}\,|\zeta|^{2}+\varepsilon(1+|z|^{2})^{\frac{p-2}{2}}|\zeta|^{2}.
\]
On the other hand, using the upper bound in \eqref{cs}, for $p\geq2$
we deduce 
\[
\langle D^{2}G_{\varepsilon}(z)\,\zeta,\zeta\rangle\le(p-1)(|z|-\lambda)_{+}^{p-2}\,|\zeta|^{2}+\varepsilon(p-1)(1+|z|^{2})^{\frac{p-2}{2}}|\zeta|^{2},
\]
where we have also used that 
\begin{equation}
h(t)+th'(t)=(p-1)(t-\lambda)_{+}^{p-2}\label{h_eq}
\end{equation}
and 
\[
a(t)+ta'(t)\le(p-1)(1+t^{2})^{\frac{p-2}{2}}\,\,\,\,\,\,\,\,\text{when}\,\,p\geq2.
\]
Otherwise, if {$h'(\vert z\vert)+\varepsilon a'(\vert z\vert)<0$}
(which may happen only when $1<p<2$), we easily get 
\[
\langle D^{2}G_{\varepsilon}(z)\,\zeta,\zeta\rangle\le\big[h(|z|)+\varepsilon a(|z|)\big]\,|\zeta|^{2}.
\]
For the derivation of the lower bound, we use the right inequality
in \eqref{cs} to deduce that 
\begin{eqnarray*}
\langle D^{2}G_{\varepsilon}(z)\zeta,\zeta\rangle & \ge & \big[h(|z|)+\varepsilon a(|z|)\big]|\zeta|^{2}+\big[h'(|z|)+\varepsilon a'(|z|)\big]\,|z|\,|\zeta|^{2}\\
\\
 & = & \big[h(|z|)+h'(|z|)|z|\big]|\zeta|^{2}+\big[a(|z|)+\varepsilon a'(|z|)|z|\big]\,|\zeta|^{2}\\
\\
 & \ge & (p-1)(|z|-\lambda)_{+}^{p-2}|\zeta|^{2}+\varepsilon(p-1)(1+|z|^{2})^{\frac{p-2}{2}}|\zeta|^{2},
\end{eqnarray*}
where, in the last line, we have used \eqref{h_eq} and the fact that
\[
a(t)+ta'(t)\ge(p-1)(1+t^{2})^{\frac{p-2}{2}}\,\,\,\,\,\,\,\,\text{when}\,\,1<p<2.
\]
This proves the claim.
\end{proof}
\noindent $\hspace*{1em}$In what follows, $u_{\varepsilon}\in u+W_{0}^{1,p}(B_{R})$
will be the unique weak solution to (\ref{eq:approximation}). By
standard elliptic regularity \cite[Chapter 8]{Giu}, we know that
$(1+|Du_{\varepsilon}|^{2})^{\frac{p-2}{4}}Du_{\varepsilon}\in W_{loc}^{1,2}(B_{R})$
and therefore $u_{\varepsilon}\in W_{loc}^{2,2}(B_{R})$. Moreover,
as usual, we shall denote by $p^{*}$ the Sobolev conjugate exponent
of $p$, defined as 
\[
p^{*}:=\begin{cases}
\frac{np}{n-p} & \text{if}\,\,p<n,\\
\text{any value in \ensuremath{(p,\infty)}} & \text{if}\,\,p\ge n,
\end{cases}
\]
and denote by $(p^{*})'$ its Hölder conjugate exponent.\\
 $\hspace*{1em}$The proofs of Theorems \ref{thm:theo1} and \ref{thm:theo2}
are crucially based on the following results. 
\begin{prop}[\textbf{Uniform energy estimate}]
 \label{prop:uniform} With the notation and under the assumptions
above, if $f\in L^{(p^{*})'}(B_{R})$, there exist two positive constants
$\varepsilon_{0}\leq1$ and $C\equiv C(n,p)$ such that 
\begin{equation}
\int_{B_{R}}\vert Du_{\varepsilon}\vert^{p}\,dx\,\leq\,C\left[1+\lambda^{p}\,+\,\Vert Du\Vert_{L^{p}(B_{R})}^{p}\,+\Vert f\Vert_{L^{(p^{*})'}(B_{R})}^{p'}\right]\label{eq:unifenest}
\end{equation}
for all $\varepsilon\in(0,\varepsilon_{0}]$. 
\end{prop}

\begin{proof}[\textbf{{Proof}}]
We insert in the weak formulation of (\ref{eq:approximation}) 
\[
\int_{B_{R}}\langle DG_{\varepsilon}(Du_{\varepsilon}),D\varphi\rangle\,dx\,=\,\int_{B_{R}}f_{\varepsilon}\,\varphi\,dx\,\,\,\,\,\,\,\,\,\,\mathrm{for\,\,every}\,\,\varphi\in W_{0}^{1,p}(B_{R}),
\]
the test function $\varphi=u_{\varepsilon}-u$. Recalling (\ref{eq:DG}),
this gives 
\begin{align}
 & \int_{B_{R}}\langle H_{p-1}(Du_{\varepsilon})+\varepsilon\,(1+\vert Du_{\varepsilon}\vert^{2})^{\frac{p-2}{2}}\,Du_{\varepsilon},Du_{\varepsilon}\rangle\,dx\label{eq4.4}\\
 & \,\,\,\,\,\,\,=\int_{B_{R}}\langle H_{p-1}(Du_{\varepsilon})+\varepsilon\,(1+\vert Du_{\varepsilon}\vert^{2})^{\frac{p-2}{2}}\,Du_{\varepsilon},Du\rangle\,dx\,+\int_{B_{R}}f_{\varepsilon}\,(u_{\varepsilon}-u)\,dx.\nonumber 
\end{align}
Since 
\[
\langle H_{p-1}(z)+\varepsilon\,(1+\vert z\vert^{2})^{\frac{p-2}{2}}z,z\rangle\,\geq\,(\vert z\vert-\lambda)_{+}^{p}\,+\,\varepsilon\,(1+\vert z\vert^{2})^{\frac{p-2}{2}}\vert z\vert^{2}\,\geq\,(\vert z\vert-\lambda)_{+}^{p}
\]
for every $z\in\mathbb{R}^{n}$, we can estimate the integrals in
\eqref{eq4.4}, thus obtaining 
\begin{align*}
 & \int_{B_{R}}(\vert Du_{\varepsilon}\vert-\lambda)_{+}^{p}\,dx\\
 & \,\,\,\,\,\,\,\leq\,\int_{B_{R}}\vert Du_{\varepsilon}\vert^{p-1}\,\vert Du\vert\,dx\,+\,\varepsilon\int_{B_{R}}(1+\vert Du_{\varepsilon}\vert^{2})^{\frac{p-1}{2}}\,\vert Du\vert\,dx\,+\,\Vert f_{\varepsilon}\Vert_{L^{(p^{*})'}(B_{R})}\,\Vert u_{\varepsilon}-u\Vert_{L^{p^{*}}(B_{R})}\\
 & \,\,\,\,\,\,\,\leq\left(1+2^{\frac{p-1}{2}}\right)\int_{B_{R}}\vert Du_{\varepsilon}\vert^{p-1}\,\vert Du\vert\,dx\,+\,2^{\frac{p-1}{2}}\int_{B_{R}}\vert Du\vert\,dx\\
 & \,\,\,\,\,\,\,\,\,\,\,\,\,\,+\,c(n,p)\,\Vert f_{\varepsilon}\Vert_{L^{(p^{*})'}(B_{R})}\,\Vert Du_{\varepsilon}-Du\Vert_{L^{p}(B_{R})}\,,
\end{align*}
where we have used Hölder's and Sobolev inequalities and the fact
that $\varepsilon,R\le1$. Now, applying Young's inequality with $\sigma>0$,
we arrive at 
\begin{align*}
 & \int_{B_{R}}(\vert Du_{\varepsilon}\vert-\lambda)_{+}^{p}\,dx\\
 & \,\,\,\,\,\,\,\leq\,\sigma\int_{B_{R}}\vert Du_{\varepsilon}\vert^{p}\,dx\,+\,c(n,p,\sigma)\int_{B_{R}}\vert Du\vert^{p}\,dx\,+\,c(n,p,\sigma)\left[1+\,\Vert f_{\varepsilon}\Vert_{L^{(p^{*})'}(B_{R})}^{p'}\right],
\end{align*}
where we have used again that $R\le1$. Since 
\begin{equation}
f_{\varepsilon}\rightarrow f\,\,\,\,\,\,\,\,\,\mathrm{strongly\,\,in\,\,}L^{(p^{*})'}(B_{R})\,\,\,\,\,\mathrm{as}\,\,\varepsilon\rightarrow0^{+},\label{eq:strongconv}
\end{equation}
there exists a positive number $\varepsilon_{0}\leq1$ such that 
\[
\Vert f_{\varepsilon}\Vert_{L^{(p^{*})'}(B_{R})}\,\leq\,1+\Vert f\Vert_{L^{(p^{*})'}(B_{R})}\,\,\,\,\,\,\,\,\mathrm{for\,\,all\,\,}\varepsilon\in(0,\varepsilon_{0}].
\]
Then, for $\varepsilon\in(0,\varepsilon_{0}]$, we have 
\begin{align}
 & \int_{B_{R}}(\vert Du_{\varepsilon}\vert-\lambda)_{+}^{p}\,dx\nonumber \\
 & \,\,\,\,\,\,\,\leq\,\sigma\int_{B_{R}}\vert Du_{\varepsilon}\vert^{p}\,dx\,+\,c\,\Vert Du\Vert_{L^{p}(B_{R})}^{p}\,+\,c\left[1+\Vert f\Vert_{L^{(p^{*})'}(B_{R})}^{p'}\right]\nonumber \\
 & \,\,\,\,\,\,\,\leq\,\sigma\int_{B_{R}}[\lambda+(\vert Du_{\varepsilon}\vert-\lambda)_{+}]^{p}\,dx\,+\,c\,\Vert Du\Vert_{L^{p}(B_{R})}^{p}+\,c\left[1+\Vert f\Vert_{L^{(p^{*})'}(B_{R})}^{p'}\right]\nonumber \\
 & \,\,\,\,\,\,\,\leq\,2^{p-1}\sigma\int_{B_{R}}(\vert Du_{\varepsilon}\vert-\lambda)_{+}^{p}\,dx\,+\,c\,\lambda^{p}\,+\,c\,\Vert Du\Vert_{L^{p}(B_{R})}^{p}\,+\,c\left[1+\Vert f\Vert_{L^{(p^{*})'}(B_{R})}^{p'}\right],\label{eq:est1}
\end{align}
where $c\equiv c(n,p,\sigma)>0$. Choosing $\sigma=\frac{1}{2^{p}}$
and absorbing the first term on the right-hand side of \eqref{eq:est1}
into the left-hand side, we obtain 
\[
\int_{B_{R}}(\vert Du_{\varepsilon}\vert-\lambda)_{+}^{p}\,dx\,\leq\,C\left[1+\lambda^{p}\,+\,\Vert Du\Vert_{L^{p}(B_{R})}^{p}\,+\,\Vert f\Vert_{L^{(p^{*})'}(B_{R})}^{p'}\right]
\]
for some finite positive constant $C$ depending on $n$ and $p$,
but not on $\varepsilon$. This estimate is sufficient to ensure the
validity of the assertion. 
\end{proof}
\begin{prop}[\textbf{Comparison estimate}]
\label{prop:confronto}With the notation and under the assumptions
above, if $f\in L^{(p^{*})'}(B_{R})$, there exists a positive constant
$C$ depending only on $n$ and $p$ such that the estimate 
\begin{align}
 & \int_{B_{R}}\left|\mathcal{V}_{\lambda}(Du_{\varepsilon})-\mathcal{V}_{\lambda}(Du)\right|^{2}dx\nonumber \\
 & \,\,\,\,\,\,\,\leq\,C\,\Vert f_{\varepsilon}-f\Vert_{L^{(p^{*})'}(B_{R})}\left[1+\lambda+\,\Vert Du\Vert_{L^{p}(B_{R})}\,+\,\Vert f\Vert_{L^{(p^{*})'}(B_{R})}^{\frac{1}{p-1}}\right]\nonumber \\
 & \,\,\,\,\,\,\,\,\,\,\,\,\,\,+\,C\,\varepsilon\left[1+\lambda^{p}\,+\,\Vert Du\Vert_{L^{p}(B_{R})}^{p}\,+\,\Vert f\Vert_{L^{(p^{*})'}(B_{R})}^{p'}\right]\label{eq:comparison}
\end{align}
holds for every $\varepsilon\in(0,\varepsilon_{0}]$, where $\varepsilon_{0}$
is the constant from Proposition \ref{prop:uniform}.
\end{prop}

\begin{proof}[\textbf{{Proof}}]
We proceed by testing equations (\ref{eq:degenerate}) and $(\ref{eq:approximation})_{1}$
with the map $\varphi=u_{\varepsilon}-u$. Thus we find 
\begin{align}
 & \int_{B_{R}}\langle H_{p-1}(Du_{\varepsilon})-H_{p-1}(Du),Du_{\varepsilon}-Du\rangle\,dx\,+\,\varepsilon\int_{B_{R}}\langle(1+\vert Du_{\varepsilon}\vert^{2})^{\frac{p-2}{2}}\,Du_{\varepsilon},Du_{\varepsilon}-Du\rangle\,dx\nonumber \\
 & \,\,\,\,\,\,\,=\int_{B_{R}}(f_{\varepsilon}-f)(u_{\varepsilon}-u)\,dx.\label{eq:001}
\end{align}
Using Lemma \ref{lem:vital}, the Cauchy-Schwarz inequality as well
as Hölder's and Young's inequalities, from \eqref{eq:001} we obtain
\begin{align*}
 & C\int_{B_{R}}\left|\mathcal{V}_{\lambda}(Du_{\varepsilon})-\mathcal{V}_{\lambda}(Du)\right|^{2}dx\,+\,\varepsilon\int_{B_{R}}(1+\vert Du_{\varepsilon}\vert^{2})^{\frac{p-2}{2}}\,\vert Du_{\varepsilon}\vert^{2}\,dx\\
 & \,\,\,\,\,\,\,\leq\,\Vert f_{\varepsilon}-f\Vert_{L^{(p^{*})'}(B_{R})}\,\Vert u_{\varepsilon}-u\Vert_{L^{p^{*}}(B_{R})}\,+\,\varepsilon\int_{B_{R}}(1+\vert Du_{\varepsilon}\vert^{2})^{\frac{p-2}{2}}\,\vert Du_{\varepsilon}\vert\,\vert Du\vert\,dx\\
 & \,\,\,\,\,\,\,\leq\,\Vert f_{\varepsilon}-f\Vert_{L^{(p^{*})'}(B_{R})}\,\Vert u_{\varepsilon}-u\Vert_{L^{p^{*}}(B_{R})}\,+\,\varepsilon\int_{B_{R}}(1+\vert Du_{\varepsilon}\vert^{2})^{\frac{p-1}{2}}\,\vert Du\vert\,dx\\
 & \,\,\,\,\,\,\,\leq\,\Vert f_{\varepsilon}-f\Vert_{L^{(p^{*})'}(B_{R})}\,\Vert u_{\varepsilon}-u\Vert_{L^{p^{*}}(B_{R})}\,+\,\frac{\varepsilon}{p'}\int_{B_{R}}(1+\vert Du_{\varepsilon}\vert^{2})^{\frac{p}{2}}\,dx\,+\,\frac{\varepsilon}{p}\,\Vert Du\Vert_{L^{p}(B_{R})}^{p}\,,
\end{align*}
where $C$ is a positive constant depending only on $n$ and $p$.
Now, let us consider the same $\varepsilon_{0}\in(0,1]$ as in Proposition
\ref{prop:uniform} and let $\varepsilon\in(0,\varepsilon_{0}]$.
Then, applying Sobolev's and Minkowski's inequalities, we get\begin{align*}
&\int_{B_{R}}\left|\mathcal{V}_{\lambda}(Du_{\varepsilon})-\mathcal{V}_{\lambda}(Du)\right|^{2}dx\\
&\,\,\,\,\,\,\,\leq\,C\,\Vert f_{\varepsilon}-f\Vert_{L^{(p^{*})'}(B_{R})}\,\Vert Du_{\varepsilon}-Du\Vert_{L^{p}(B_{R})}\,+\,\frac{C\,\varepsilon}{p'}\int_{B_{R}}(1+\vert Du_{\varepsilon}\vert^{2})^{\frac{p}{2}}\,dx\,+\,\frac{C\,\varepsilon}{p}\,\Vert Du\Vert_{L^{p}(B_{R})}^{p}\\
&\,\,\,\,\,\,\,\leq\,C\,\Vert f_{\varepsilon}-f\Vert_{L^{(p^{*})'}(B_{R})}\left(\Vert Du_{\varepsilon}\Vert_{L^{p}(B_{R})}\,+\,\Vert Du\Vert_{L^{p}(B_{R})}\right)\\
&\,\,\,\,\,\,\,\,\,\,\,\,\,\,+\,\frac{C\,\varepsilon}{p'}\int_{B_{R}}(1+\vert Du_{\varepsilon}\vert^{2})^{\frac{p}{2}}\,dx\,+\,\frac{C\,\varepsilon}{p}\,\Vert Du\Vert_{L^{p}(B_{R})}^{p}\\
&\,\,\,\,\,\,\,\leq\,C\,\Vert f_{\varepsilon}-f\Vert_{L^{(p^{*})'}(B_{R})}\left[1+\lambda+\,\Vert Du\Vert_{L^{p}(B_{R})}\,+\,\Vert f\Vert_{L^{(p^{*})'}(B_{R})}^{\frac{1}{p-1}}\right]\\
&\,\,\,\,\,\,\,\,\,\,\,\,\,\,+\,C\,\varepsilon\left[1+\lambda^{p}\,+\,\Vert Du\Vert_{L^{p}(B_{R})}^{p}\,+\,\Vert f\Vert_{L^{(p^{*})'}(B_{R})}^{p'}\right],
\end{align*}where, in the last two lines, we have used inequality (\ref{eq:unifenest}).
This concludes the proof.
\end{proof}

\section{Proof of Theorem \ref{thm:theo1} \label{sec:theo1}}

$\hspace*{1em}$In this section we prove Theorem \ref{thm:theo1},
by dividing the proof into two steps. First, we shall derive a suitable
uniform \textit{a priori} estimate for the weak solutions $u_{\varepsilon}$
of the regularized problems \eqref{eq:approximation}. Then, we conclude
with a standard comparison argument (see e.g.\ \cite{AmPa,DMS,GPdN})
which, combined with the estimates from Propositions \ref{prop:uniform},
\ref{prop:confronto} and \ref{prop:uniform2}, yields the local Sobolev
regularity of the function $\mathcal{V}_{\lambda}(Du)$. We
begin with the following result. 
\selectlanguage{british}%
\begin{prop}[\textbf{Uniform Sobolev estimate}]
\foreignlanguage{english}{ \label{prop:uniform2} }Under the assumptions
of Theorem \ref{thm:theo1} and with the notation above, there exists
a positive number $\varepsilon_{1}\leq1$ such that, for every $\varepsilon\in(0,\varepsilon_{1}]$
and every pair of concentric balls $B_{r/2}\subset B_{r}\subset B_{R}$,
we have\begin{align}\label{eq:Sobolev1}
&\int_{B_{r/2}}\vert D\mathcal{V}_{\lambda}(Du_{\varepsilon})\vert^{2}\,dx\nonumber\\
&\,\,\,\,\,\,\,\leq\left(C+\,\frac{C}{r^{2}}\right)\left[1+\lambda^{p}+\,\Vert Du\Vert_{L^{p}(B_{R})}^{p}\,+\Vert f\Vert_{L^{p'}(B_{R})}^{p'}\right]+\,C\,\Vert f\Vert_{B_{p',1}^{\frac{p-2}{p}}(B_{R})}^{p'}
\end{align} 
for a positive constant $C$ depending only on $n$ and $p$.
\end{prop}

\selectlanguage{english}%
\begin{proof}[\textbf{{Proof}}]
 Let us first assume that $\lambda>0$. Differentiating the equation
in (\ref{eq:approximation}) with respect to $x_{j}$ for some $j\in\{1,\dots,n\}$
and then integrating by parts, we obtain 
\begin{equation}
\int_{B_{R}}\langle D^{2}G_{\varepsilon}(Du_{\varepsilon})\,D(\partial_{j}u_{\varepsilon}),D\varphi\rangle\,dx\,=\,\int_{B_{R}}(\partial_{j}f_{\varepsilon})\,\varphi\,dx,\qquad\forall\,\varphi\in W_{0}^{1,p}(B_{R}).\label{secondvar}
\end{equation}
Let $\eta\in C_{0}^{\infty}(B_{r})$ be a standard cut-off function
such that 
\begin{equation}
0\le\eta\leq1,\,\,\,\,\,\,\eta\equiv1\,\,\,\mathrm{on}\,\,\overline{B}_{r/2}\,,\,\,\,\,\,\,\Vert D\eta\Vert_{\infty}\leq\,\frac{\tilde{c}}{r}\,,\label{eq:eta}
\end{equation}
and choose 
\[
\varphi\,=\,\eta^{2}\,(\partial_{j}u_{\varepsilon})\,\Phi\big((\vert Du_{\varepsilon}\vert-\lambda)_{+}\big),
\]
where $\Phi:[0,\infty)\rightarrow[0,\infty)$ is an increasing, locally
Lipschitz continuous function, such that $\Phi$ and $\Phi'$ are
bounded on $[0,\infty)$, $\Phi(0)=0$ and 
\begin{equation}
\Phi'(t)\,t\,\leq\,c_{\Phi}\,\Phi(t)\label{eq:derivative}
\end{equation}
for a suitable constant $c_{\Phi}>0$. Using the above choice of $\varphi$
as a test function in \eqref{secondvar}, we get 
\begin{equation}
\begin{split} & \int_{B_{r}}\langle D^{2}G_{\varepsilon}(Du_{\varepsilon})\,D(\partial_{j}u_{\varepsilon}),D(\partial_{j}u_{\varepsilon})\rangle\,\eta^{2}\,\Phi\left((\vert Du_{\varepsilon}\vert-\lambda)_{+}\right)\,dx\\
 & +\int_{B_{r}}\langle D^{2}G_{\varepsilon}(Du_{\varepsilon})\,D(\partial_{j}u_{\varepsilon}),D[(\vert Du_{\varepsilon}\vert-\lambda)_{+}]\rangle\,\eta^{2}\,(\partial_{j}u_{\varepsilon})\,\Phi'\left((\vert Du_{\varepsilon}\vert-\lambda)_{+}\right)\,dx\\
 & \,\,\,\,\,\,\,=\,-\,2\int_{B_{r}}\langle D^{2}G_{\varepsilon}(Du_{\varepsilon})\,D(\partial_{j}u_{\varepsilon}),D\eta\rangle\,\eta\,(\partial_{j}u_{\varepsilon})\,\Phi\left((\vert Du_{\varepsilon}\vert-\lambda)_{+}\right)\,dx\\
 & \,\,\,\,\,\,\,\,\,\,\,\,\,\,+\int_{B_{r}}(\partial_{j}f_{\varepsilon})(\partial_{j}u_{\varepsilon})\,\eta^{2}\,\Phi\left((\vert Du_{\varepsilon}\vert-\lambda)_{+}\right)\,dx.
\end{split}
\label{eq:equa1}
\end{equation}
As for the first term on the right-hand side of \eqref{eq:equa1},
we have 
\begin{align}
 & -\,2\int_{B_{r}}\langle D^{2}G_{\varepsilon}(Du_{\varepsilon})\,D(\partial_{j}u_{\varepsilon}),D\eta\rangle\,\eta\,(\partial_{j}u_{\varepsilon})\,\Phi\left((\vert Du_{\varepsilon}\vert-\lambda)_{+}\right)\,dx\nonumber \\
 & \,\,\,\,\,\,\,\leq\,2\int_{B_{r}}\sqrt{\langle D^{2}G_{\varepsilon}(Du_{\varepsilon})\,D(\partial_{j}u_{\varepsilon}),D(\partial_{j}u_{\varepsilon})\rangle}\,\sqrt{\langle D^{2}G_{\varepsilon}(Du_{\varepsilon})\,D\eta,D\eta\rangle}\,\eta\,\vert\partial_{j}u_{\varepsilon}\vert\,\Phi\left((\vert Du_{\varepsilon}\vert-\lambda)_{+}\right)\,dx\nonumber \\
 & \,\,\,\,\,\,\,\leq\,\frac{1}{2}\int_{B_{r}}\langle D^{2}G_{\varepsilon}(Du_{\varepsilon})\,D(\partial_{j}u_{\varepsilon}),D(\partial_{j}u_{\varepsilon})\rangle\,\eta^{2}\,\Phi\left((\vert Du_{\varepsilon}\vert-\lambda)_{+}\right)\,dx\nonumber \\
 & \,\,\,\,\,\,\,\,\,\,\,\,\,\,+\,\,2\int_{B_{r}}\langle D^{2}G_{\varepsilon}(Du_{\varepsilon})\,D\eta,D\eta\rangle\,\vert\partial_{j}u_{\varepsilon}\vert^{2}\,\Phi\left((\vert Du_{\varepsilon}\vert-\lambda)_{+}\right)\,dx,%
\label{eq:equa2}
\end{align}
where we have used Cauchy-Schwarz and Young's inequalities. Joining
\eqref{eq:equa1} and \eqref{eq:equa2}, we get 
\begin{align}
 & \int_{B_{r}}\langle D^{2}G_{\varepsilon}(Du_{\varepsilon})\,D(\partial_{j}u_{\varepsilon}),D(\partial_{j}u_{\varepsilon})\rangle\,\eta^{2}\,\Phi\left((\vert Du_{\varepsilon}\vert-\lambda)_{+}\right)\,dx\nonumber \\
 & +\int_{B_{r}}\langle D^{2}G_{\varepsilon}(Du_{\varepsilon})\,D(\partial_{j}u_{\varepsilon}),D[(\vert Du_{\varepsilon}\vert-\lambda)_{+}]\rangle\,\eta^{2}\,(\partial_{j}u_{\varepsilon})\,\Phi'\left((\vert Du_{\varepsilon}\vert-\lambda)_{+}\right)\,dx\nonumber \\
 & \,\,\,\,\,\,\,\le\,\frac{1}{2}\int_{B_{r}}\langle D^{2}G_{\varepsilon}(Du_{\varepsilon})\,D(\partial_{j}u_{\varepsilon}),D(\partial_{j}u_{\varepsilon})\rangle\,\eta^{2}\,\Phi\left((\vert Du_{\varepsilon}\vert-\lambda)_{+}\right)\,dx\nonumber \\
 & \,\,\,\,\,\,\,\,\,\,\,\,\,\,+\,\,2\int_{B_{r}}\langle D^{2}G_{\varepsilon}(Du_{\varepsilon})\,D\eta,D\eta\rangle\,\vert\partial_{j}u_{\varepsilon}\vert^{2}\,\Phi\left((\vert Du_{\varepsilon}\vert-\lambda)_{+}\right)\,dx\nonumber \\
 & \,\,\,\,\,\,\,\,\,\,\,\,\,\,+\int_{B_{r}}(\partial_{j}f_{\varepsilon})(\partial_{j}u_{\varepsilon})\,\eta^{2}\,\Phi\left((\vert Du_{\varepsilon}\vert-\lambda)_{+}\right)\,dx.\label{eq:mod1}
\end{align}
Reabsorbing the first integral in the right-hand side of \eqref{eq:mod1}
by the left-hand side and summing the resulting inequalities with
respect to $j$ from $1$ to $n$, we obtain 
\begin{equation}
I_{1}+I_{2}\,\leq\,I_{3}+I_{4}\,,\label{eq:integrals}
\end{equation}
where 
\begin{align*}
 & I_{1}:=\int_{B_{r}}\sum_{j=1}^{n}\,\langle D^{2}G_{\varepsilon}(Du_{\varepsilon})\,D(\partial_{j}u_{\varepsilon}),D(\partial_{j}u_{\varepsilon})\rangle\,\eta^{2}\,\Phi\left((\vert Du_{\varepsilon}\vert-\lambda)_{+}\right)\,dx,\\
 & I_{2}:=\,2\int_{B_{r}}\sum_{j=1}^{n}\,\langle D^{2}G_{\varepsilon}(Du_{\varepsilon})\,D(\partial_{j}u_{\varepsilon}),D[(\vert Du_{\varepsilon}\vert-\lambda)_{+}]\rangle\,\eta^{2}\,(\partial_{j}u_{\varepsilon})\,\Phi'\left((\vert Du_{\varepsilon}\vert-\lambda)_{+}\right)\,dx,\\
 & I_{3}:=\,4\int_{B_{r}}\sum_{j=1}^{n}\langle D^{2}G_{\varepsilon}(Du_{\varepsilon})\,D\eta,D\eta\rangle\,\vert\partial_{j}u_{\varepsilon}\vert^{2}\,\Phi\left((\vert Du_{\varepsilon}\vert-\lambda)_{+}\right)\,dx,\\
 & I_{4}:=\,2\int_{B_{r}}\sum_{j=1}^{n}\,(\partial_{j}f_{\varepsilon})(\partial_{j}u_{\varepsilon})\,\eta^{2}\,\Phi\left((\vert Du_{\varepsilon}\vert-\lambda)_{+}\right)\,dx.
\end{align*}
We now prove that $I_{2}$ is non-negative, thus we can drop
it in the following. Recalling the definitions \eqref{eq:minaut} and \eqref{eq:maxaut}, for $\vert Du_{\varepsilon}\vert>\lambda$
we have 
\begin{align}
&\sum_{j=1}^{n}\,\langle D^{2}G_{\varepsilon}(Du_{\varepsilon})\,D(\partial_{j}u_{\varepsilon}),D[(\vert Du_{\varepsilon}\vert-\lambda)_{+}]\rangle\,(\partial_{j}u_{\varepsilon})\nonumber\\
 & \,\,\,\,\,\,\,=\left[\frac{\boldsymbol{\Lambda}(\vert Du_{\varepsilon}\vert)}{\vert Du_{\varepsilon}\vert^{2}}\,-\,\frac{\boldsymbol{\lambda}(\vert Du_{\varepsilon}\vert)}{\vert Du_{\varepsilon}\vert^{2}}\,+\,\varepsilon\,(p-2)\,(1+\vert Du_{\varepsilon}\vert^{2})^{\frac{p-4}{2}}\right]\nonumber\\
 & \,\,\,\,\,\,\,\,\,\,\,\,\,\,\cdot\sum_{i,\,j,\,k\,=\,1}^{n}\,(\partial_{j}u_{\varepsilon})(\partial_{i}u_{\varepsilon})(\partial_{k}u_{\varepsilon})(\partial_{ij}^{2}u_{\varepsilon})\,\partial_{k}[(\vert Du_{\varepsilon}\vert-\lambda)_{+}]\nonumber\\
 & \,\,\,\,\,\,\,\,\,\,\,\,\,\,+\left[\boldsymbol{\lambda}(\vert Du_{\varepsilon}\vert)\,+\,\varepsilon\,(1+\vert Du_{\varepsilon}\vert^{2})^{\frac{p-2}{2}}\right]\sum_{i,\,j\,=\,1}^{n}\,(\partial_{j}u_{\varepsilon})(\partial_{ij}^{2}u_{\varepsilon})\,\partial_{i}[(\vert Du_{\varepsilon}\vert-\lambda)_{+}]\nonumber\\
 & \,\,\,\,\,\,\,=\left[\frac{\boldsymbol{\Lambda}(\vert Du_{\varepsilon}\vert)}{\vert Du_{\varepsilon}\vert}\,-\,\frac{\boldsymbol{\lambda}(\vert Du_{\varepsilon}\vert)}{\vert Du_{\varepsilon}\vert}\,+\,\varepsilon\,(p-2)\,(1+\vert Du_{\varepsilon}\vert^{2})^{\frac{p-4}{2}}\,\vert Du_{\varepsilon}\vert\right]\nonumber\\
 & \,\,\,\,\,\,\,\,\,\,\,\,\,\,\cdot\left[\sum_{k=1}^{n}\,(\partial_{k}u_{\varepsilon})\,\partial_{k}[(\vert Du_{\varepsilon}\vert-\lambda)_{+}]\right]^{2}\nonumber\\
 & \,\,\,\,\,\,\,\,\,\,\,\,\,\,+\left[\boldsymbol{\lambda}(\vert Du_{\varepsilon}\vert)\vert Du_{\varepsilon}\vert\,+\,\varepsilon\,(1+\vert Du_{\varepsilon}\vert^{2})^{\frac{p-2}{2}}\,\vert Du_{\varepsilon}\vert\right]\,\big\vert D[(\vert Du_{\varepsilon}\vert-\lambda)_{+}]\big\vert^{2},
\label{eq:equa3}
\end{align}
where we have used the fact that 
\[
\partial_{k}[(\vert Du_{\varepsilon}\vert-\lambda)_{+}]\,=\,\partial_{k}(\vert Du_{\varepsilon}\vert)\,=\,\frac{1}{\vert Du_{\varepsilon}\vert}\,\sum_{j=1}^{n}(\partial_{j}u_{\varepsilon})(\partial_{kj}^{2}u_{\varepsilon})\,\,\,\,\,\,\,\,\,\,\,\mathrm{when}\,\,\vert Du_{\varepsilon}\vert>\lambda.
\]
Thus, coming back to the estimate of $I_{2}$, from \eqref{eq:equa3}
we deduce 
\begin{align*}
I_{2}\,\geq\,2\bigintssss_{B_{r}}\eta^{2}\,\Phi' \left((\vert Du_{\varepsilon}\vert-\lambda)_{+}\right)\Bigg\{ & \left[\frac{\boldsymbol{\Lambda}(\vert Du_{\varepsilon}\vert)}{\vert Du_{\varepsilon}\vert}\,-\,\frac{\boldsymbol{\lambda}(\vert Du_{\varepsilon}\vert)}{\vert Du_{\varepsilon}\vert}\right]\cdot\left[\sum_{k=1}^{n}\,(\partial_{k}u_{\varepsilon})\,\partial_{k}[(\vert Du_{\varepsilon}\vert-\lambda)_{+}]\right]^{2}\\
 & +\,\boldsymbol{\lambda}(\vert Du_{\varepsilon}\vert)\vert Du_{\varepsilon}\vert\,\big\vert D[(\vert Du_{\varepsilon}\vert-\lambda)_{+}]\big\vert^{2}\Bigg\}\,dx.
\end{align*}
Now, arguing as in the proof of \cite[Lemma 4.1]{Marc}, for $\vert Du_{\varepsilon}\vert>\lambda$
we have 
\begin{equation}
\left[\sum_{k=1}^{n}\,(\partial_{k}u_{\varepsilon})\,\partial_{k}[(\vert Du_{\varepsilon}\vert-\lambda)_{+}]\right]^{2}\leq\,\vert Du_{\varepsilon}\vert^{2}\,\big\vert D[(\vert Du_{\varepsilon}\vert-\lambda)_{+}]\big\vert^{2}.\label{eq:dersec}
\end{equation}
This implies 
\[
I_{2}\,\geq\,2\bigintssss_{B_{r}}\eta^{2}\,\Phi'\left((\vert Du_{\varepsilon}\vert-\lambda)_{+}\right)\,\frac{\boldsymbol{\Lambda}(\vert Du_{\varepsilon}\vert)}{\vert Du_{\varepsilon}\vert}\left[\sum_{k=1}^{n}\,(\partial_{k}u_{\varepsilon})\,\partial_{k}[(\vert Du_{\varepsilon}\vert-\lambda)_{+}]\right]^{2}dx\,\geq\,0,
\]
where we have used the fact that $\Phi'\left((\vert Du_{\varepsilon}\vert-\lambda)_{+}\right)\geq0$. Thus, inequality (\ref{eq:integrals}) boils down to
\begin{equation}\label{eq:stimafond}
I_{1}\,\leq\,I_{3}+I_{4}\,.
\end{equation}
Now we choose
\begin{equation}
\Phi(t)\,:=\,\left(\frac{t}{t+\lambda}\right)^{1+\frac{2}{p-1}}\,\,\,\,\,\,\,\,\,\,\mathrm{for}\,\,t\geq0\,,\label{eq:Phi}
\end{equation}
and therefore 
\[
\Phi'(t)\,=\,\frac{p+1}{p-1}\cdot\frac{\lambda\, t^{\frac{2}{p-1}}}{(t+\lambda)^{2+\frac{2}{p-1}}}\,.
\]
Clearly, the function $\Phi$ in (\ref{eq:Phi}) satisfies (\ref{eq:derivative}) with $c_{\Phi}=\frac{p+1}{p-1}$. At this stage, we proceed by estimating separately the integrals in (\ref{eq:stimafond}).\\
\\
$\hspace*{1em}${\textbf{\textsc{Estimate
of $I_{1}$}}}\\
\\
$\hspace*{1em}$Applying Lemma \ref{lem:qform}, we get
\begin{equation}\label{I1_A}
I_1 \ge \int_{B_{r}}\boldsymbol{\lambda}(\vert Du_{\varepsilon}\vert)\,\vert D^{2}u_{\varepsilon}\vert^{2}\,\eta^{2}\,\Phi\left((\vert Du_{\varepsilon}\vert-\lambda)_{+}\right)\,dx\,.
\end{equation}
$\hspace*{1em}${\textbf{\textsc{Estimate
of $I_{3}$}}}\\
\\
$\hspace*{1em}$Using Lemma \ref{lem:qform}, (\ref{eq:eta}) and the fact that $\Phi\leq 1$, we infer
\begin{equation}\label{I3_A}
I_3 \le \,\frac{c(p)}{r^{2}}\int_{B_{r}}(1+\vert Du_{\varepsilon}\vert^{2})^{\frac{p}{2}}\,dx\,.
\end{equation}$\hspace*{1em}${\textbf{\textsc{Estimate
of $I_{4}$}}}
\\
\\
$\hspace*{1em}$By Theorem \ref{thm:extension}, there exists a linear and bounded
extension operator 
\[
\mathrm{ext}_{B_{r}}:B_{p',1}^{-2/p}(B_{r})\hookrightarrow B_{p',1}^{-2/p}(\mathbb{R}^{n})
\]
such that $\mathrm{re}_{B_{r}}\circ\mathrm{ext}_{B_{r}}=\mathrm{id}$,
where $\mathrm{re}_{B_{r}}$ is the restriction operator defined in
Section \ref{sec:Besov spaces} and the symbol $\mathrm{id}$ denotes
the identity in $B_{p',1}^{-2/p}(B_{r})$. Since $\partial_{j}f_{\varepsilon}=\mathrm{ext}_{B_{r}}(\partial_{j}f_{\varepsilon})$
almost everywhere in $B_{r}$, we have 
\[
\int_{B_{r}}(\partial_{j}f_{\varepsilon})(\partial_{j}u_{\varepsilon})\,\eta^{2}\,\Phi((\vert Du_{\varepsilon}\vert-\lambda)_{+})\,dx\,=\int_{B_{r}}\mathrm{ext}_{B_{r}}(\partial_{j}f_{\varepsilon})\cdot(\partial_{j}u_{\varepsilon})\,\eta^{2}\,\Phi((\vert Du_{\varepsilon}\vert-\lambda)_{+})\,dx.
\]
At this point, we need to estimate the integral containing $\mathrm{ext}_{B_{r}}(\partial_{j}f_{\varepsilon})$.
To this aim, we argue as in \cite[Proposition 3.2]{BS}. By definition
of dual norm, we get
\begin{align}\label{eq:dualinequality}
&\left|\int_{B_{r}}\mathrm{ext}_{B_{r}}(\partial_{j}f_{\varepsilon})\cdot(\partial_{j}u_{\varepsilon})\,\eta^{2}\,\Phi((\vert Du_{\varepsilon}\vert-\lambda)_{+})\,dx\right|\nonumber\\
&\,\,\,\,\,\,\,\leq\,\Vert\mathrm{ext}_{B_{r}}(\partial_{j}f_{\varepsilon})\Vert_{(\mathring{B}_{p,\infty}^{2/p}(\mathbb{R}^{n}))'}\,\Vert(\partial_{j}u_{\varepsilon})\,\eta^{2}\,\Phi((\vert Du_{\varepsilon}\vert-\lambda)_{+})\Vert_{B_{p,\infty}^{2/p}(\mathbb{R}^{n})}\nonumber\\
&\,\,\,\,\,\,\,=\,\Vert\mathrm{ext}_{B_{r}}(\partial_{j}f_{\varepsilon})\Vert_{(\mathring{B}_{p,\infty}^{2/p}(\mathbb{R}^{n}))'}\nonumber\\
&\,\,\,\,\,\,\,\,\,\,\,\,\,\,\cdot\left(\Vert(\partial_{j}u_{\varepsilon})\,\eta^{2}\,\Phi((\vert Du_{\varepsilon}\vert-\lambda)_{+})\Vert_{L^{p}(\mathbb{R}^{n})}\,+\left[(\partial_{j}u_{\varepsilon})\,\eta^{2}\,\Phi((\vert Du_{\varepsilon}\vert-\lambda)_{+})\right]_{B_{p,\infty}^{2/p}(\mathbb{R}^{n})}\right),
\end{align}where, in the last line, we have used the equivalence \eqref{eq:equivalence}.
By the properties of $\eta$ and the fact that $\Phi\leq1$, we have
\begin{equation}
\Vert(\partial_{j}u_{\varepsilon})\,\eta^{2}\,\Phi((\vert Du_{\varepsilon}\vert-\lambda)_{+})\Vert_{L^{p}(\mathbb{R}^{n})}\,\leq\,\Vert Du_{\varepsilon}\Vert_{L^{p}(B_{r})}\,.\label{eq:extra}
\end{equation}
Moreover, using Theorem \ref{duality}, we obtain 
\[
\Vert\mathrm{ext}_{B_{r}}(\partial_{j}f_{\varepsilon})\Vert_{(\mathring{B}_{p,\infty}^{2/p}(\mathbb{R}^{n}))'}\,\le\,c\,\Vert\mathrm{ext}_{B_{r}}(\partial_{j}f_{\varepsilon})\Vert_{B_{p',1}^{-2/p}(\mathbb{R}^{n})}\,,
\]
for some positive constant $c\equiv c(n,p)$. Combining the above
inequality and the boundedness of the operator $\mathrm{ext}_{B_{r}}$
yields 
\[
\Vert\mathrm{ext}_{B_{r}}(\partial_{j}f_{\varepsilon})\Vert_{(\mathring{B}_{p,\infty}^{2/p}(\mathbb{R}^{n}))'}\,\leq\,c\,\Vert\partial_{j}f_{\varepsilon}\Vert_{B_{p',1}^{-2/p}(B_{r})}\,.
\]
Furthermore, applying Theorem \ref{negder}, we find that 
\[
\Vert\partial_{j}f_{\varepsilon}\Vert_{B_{p',1}^{-2/p}(B_{r})}\,\le\,c\,\Vert f_{\varepsilon}\Vert_{B_{p',1}^{\frac{p-2}{p}}(B_{r})}.
\]
Combining the preceding inequalities, we infer 
\begin{equation}
\Vert\mathrm{ext}_{B_{r}}(\partial_{j}f_{\varepsilon})\Vert_{(\mathring{B}_{p,\infty}^{2/p}(\mathbb{R}^{n}))'}\,\le\,c\,\Vert f_{\varepsilon}\Vert_{B_{p',1}^{\frac{p-2}{p}}(B_{r})},\label{eq:derf}
\end{equation}
for a positive constant $c$ depending only on $n$ and $p$. Now it remains to estimate the term
\[
\left[(\partial_{j}u_{\varepsilon})\,\eta^{2}\,\Phi((\vert Du_{\varepsilon}\vert-\lambda)_{+})\right]_{B_{p,\infty}^{2/p}(\mathbb{R}^{n})}:=\,\sup_{\vert h\vert\,>\,0}\left(\int_{\mathbb{R}^{n}}\frac{\left|\delta_{h}\left((\partial_{j}u_{\varepsilon})\,\eta^{2}\,\Phi((\vert Du_{\varepsilon}\vert-\lambda)_{+})\right)\right|^{p}}{\vert h\vert^{2}}\,dx\right)^\frac{1}{p}.
\]
By applying Lemma \ref{lem:Lind}, we deduce 
\begin{align*}
 & \int_{\mathbb{R}^{n}}\frac{\left|\delta_{h}\left((\partial_{j}u_{\varepsilon})\,\eta^{2}\,\Phi((\vert Du_{\varepsilon}\vert-\lambda)_{+})\right)\right|^{p}}{\vert h\vert^{2}}\,dx\\
 & \,\,\,\,\,\,\,\leq\,\frac{c}{\vert h\vert^{2}}\int_{\mathbb{R}^{n}}\left|\delta_{h}\left(\vert\partial_{j}u_{\varepsilon}\vert^{\frac{p-2}{2}}(\partial_{j}u_{\varepsilon})\,\eta^{p}\,[\Phi((\vert Du_{\varepsilon}\vert-\lambda)_{+})]^{\frac{p}{2}}\right)\right|^{2}dx\\
 & \,\,\,\,\,\,\,\leq c(n,p)\int_{\mathbb{R}^{n}}\left|D\left(\vert\partial_{j}u_{\varepsilon}\vert^{\frac{p-2}{2}}(\partial_{j}u_{\varepsilon})\,\eta^{p}\,[\Phi((\vert Du_{\varepsilon}\vert-\lambda)_{+})]^{\frac{p}{2}}\right)\right|^{2}dx,
\end{align*}
where, in the last line, we have used the first statement in Lemma
\ref{lem:Giusti1}. By the properties of $\eta$ at (\ref{eq:eta})
and the boundedness of $\Phi$, one can easily obtain 
\begin{align}
\left[(\partial_{j}u_{\varepsilon})\,\eta^{2}\,\Phi((\vert Du_{\varepsilon}\vert-\lambda)_{+})\right]_{B_{p,\infty}^{2/p}(\mathbb{R}^{n})}^{p} & \leq\,c\int_{B_{r}}\left|D\left(\vert\partial_{j}u_{\varepsilon}\vert^{\frac{p-2}{2}}(\partial_{j}u_{\varepsilon})\,[\Phi((\vert Du_{\varepsilon}\vert-\lambda)_{+})]^{\frac{p}{2}}\right)\right|^{2}\eta^{2}\,dx\nonumber \\
 & \,\,\,\,\,\,\,+\,\,\frac{c}{r^{2}}\int_{B_{r}}\vert Du_{\varepsilon}\vert^{p}\,dx\,,\label{eq:equa9}
\end{align}
where $c\equiv c(n,p)>0$. Now, a straightforward computation reveals
that, for every $k\in\{1,\ldots,n\}$, we have 
\begin{align*}
 & \partial_{k}\left[\vert\partial_{j}u_{\varepsilon}\vert^{\frac{p-2}{2}}(\partial_{j}u_{\varepsilon})\,[\Phi((\vert Du_{\varepsilon}\vert-\lambda)_{+})]^{\frac{p}{2}}\right]\\
 & \,\,\,\,\,\,\,=\,\frac{p}{2}\,\vert\partial_{j}u_{\varepsilon}\vert^{\frac{p-2}{2}}(\partial_{kj}^{2}u_{\varepsilon})\,[\Phi((\vert Du_{\varepsilon}\vert-\lambda)_{+})]^{\frac{p}{2}}\\
 & \,\,\,\,\,\,\,\,\,\,\,\,\,\,+\,\,\frac{p}{2}\,\vert\partial_{j}u_{\varepsilon}\vert^{\frac{p-2}{2}}(\partial_{j}u_{\varepsilon})\,[\Phi((\vert Du_{\varepsilon}\vert-\lambda)_{+})]^{\frac{p-2}{2}}\,\Phi'((\vert Du_{\varepsilon}\vert-\lambda)_{+})\,\frac{\langle Du_{\varepsilon},\partial_{k}Du_{\varepsilon}\rangle}{\vert Du_{\varepsilon}\vert}\,.
\end{align*}
This yields 
\begin{equation}
\left|D\left(\vert\partial_{j}u_{\varepsilon}\vert^{\frac{p-2}{2}}(\partial_{j}u_{\varepsilon})\,[\Phi((\vert Du_{\varepsilon}\vert-\lambda)_{+})]^{\frac{p}{2}}\right)\right|^{2}\leq\,c(n,p)\,(\mathbf{A}_{1}+\mathbf{A}_{2})\,,\label{eq:A1+A2}
\end{equation}
where we set 
\[
\mathbf{A}_{1}:=\,\vert Du_{\varepsilon}\vert^{p-2}\,\vert D^{2}u_{\varepsilon}\vert^{2}\,[\Phi((\vert Du_{\varepsilon}\vert-\lambda)_{+})]^{p}
\]
and 
\[
\mathbf{A}_{2}:=\,\vert Du_{\varepsilon}\vert^{p}\,\vert D^{2}u_{\varepsilon}\vert^{2}\,[\Phi((\vert Du_{\varepsilon}\vert-\lambda)_{+})]^{p-2}\,[\Phi'((\vert Du_{\varepsilon}\vert-\lambda)_{+})]^{2}.
\]
We now estimate $\mathbf{A}_{1}$ and $\mathbf{A}_{2}$ in the set where $\vert Du_{\varepsilon}\vert>\lambda$, since both
$\mathbf{A}_{1}$ and $\mathbf{A}_{2}$ vanish in the set $\{|Du_{\varepsilon}|\le\lambda\}$. Note that, for $|Du_\varepsilon|>\lambda$, we have
\begin{align*}
\mathbf{A}_{1}&=\,\boldsymbol{\lambda}(\vert Du_{\varepsilon}\vert)\,\Phi(\vert Du_{\varepsilon}\vert-\lambda)\,\vert D^{2}u_{\varepsilon}\vert^{2}\,\frac{[\Phi(\vert Du_{\varepsilon}\vert-\lambda)]^{p-1}}{\boldsymbol{\lambda}(\vert Du_{\varepsilon}\vert)}\,\vert Du_{\varepsilon}\vert^{p-2}\\
&=\,\boldsymbol{\lambda}(\vert Du_{\varepsilon}\vert)\,\Phi(\vert Du_{\varepsilon}\vert-\lambda)\,\vert D^{2}u_{\varepsilon}\vert^{2}\left[\frac{\Phi(\vert Du_{\varepsilon}\vert-\lambda)}{\vert Du_{\varepsilon}\vert-\lambda}\,\vert Du_{\varepsilon}\vert\right]^{p-1}
\end{align*}
and 
\begin{align*}
\mathbf{A}_{2}\,&=\,\boldsymbol{\lambda}(\vert Du_{\varepsilon}\vert)\,\Phi(\vert Du_{\varepsilon}\vert-\lambda)\,\vert D^{2}u_{\varepsilon}\vert^{2}\left[\frac{\Phi(\vert Du_{\varepsilon}\vert-\lambda)}{\vert Du_{\varepsilon}\vert-\lambda}\,\vert Du_{\varepsilon}\vert\right]^{p-1}\left[\frac{\Phi'(\vert Du_{\varepsilon}\vert-\lambda)}{\Phi(\vert Du_{\varepsilon}\vert-\lambda)}\,\vert Du_{\varepsilon}\vert\right]^{2}\\
&=\mathbf{A}_{1}\left[\frac{\Phi'(\vert Du_{\varepsilon}\vert-\lambda)}{\Phi(\vert Du_{\varepsilon}\vert-\lambda)}\,\vert Du_{\varepsilon}\vert\right]^{2}.
\end{align*}
Recalling the definition of $\Phi$ in \eqref{eq:Phi}, we find that
\begin{equation*}
   \left[\frac{\Phi(\vert Du_{\varepsilon}\vert-\lambda)}{\vert Du_{\varepsilon}\vert-\lambda}\,\vert Du_{\varepsilon}\vert\right]^{p-1}=\left(\frac{\vert Du_{\varepsilon}\vert-\lambda}{|Du_\varepsilon|}\right)^{2}\le\,1\, 
\end{equation*}
and, moreover, 
\begin{equation*}
\left[\frac{\Phi'(\vert Du_{\varepsilon}\vert-\lambda)}{\Phi(\vert Du_{\varepsilon}\vert-\lambda)}\,|Du_\varepsilon|\right]^{2}=\left(\frac{p+1}{p-1}\right)^2\frac{\lambda^2}{(|Du_\varepsilon|-\lambda)^2}\,.
\end{equation*}\\
Therefore, combining the four previous estimates, for $\vert Du_\varepsilon\vert>\lambda$ we get 
\begin{equation}\label{est:A1}
\mathbf{A}_{1}\,\le\,c(p)\,\boldsymbol{\lambda}(\vert Du_{\varepsilon}\vert)\,\Phi(\vert Du_{\varepsilon}\vert-\lambda)\,\vert D^{2}u_{\varepsilon}\vert^{2}\end{equation}
and
\begin{align}\label{est:A2}
\mathbf{A}_{2}\,&=\,c(p)\,\boldsymbol{\lambda}(\vert Du_{\varepsilon}\vert)\,\Phi(\vert Du_{\varepsilon}\vert-\lambda)\,\vert D^{2}u_{\varepsilon}\vert^{2}\,\frac{(\vert Du_{\varepsilon}\vert-\lambda)^2}{\vert Du_{\varepsilon}\vert^2}\cdot\frac{\lambda^2}{(\vert Du_\varepsilon\vert-\lambda)^2}\nonumber\\
&\le\,c(p)\,\boldsymbol{\lambda}(\vert Du_{\varepsilon}\vert)\,\Phi(\vert Du_{\varepsilon}\vert-\lambda)\,\vert D^{2}u_{\varepsilon}\vert^{2}.
\end{align}
Joining estimates \eqref{eq:equa9}$-$\eqref{est:A2}, we obtain
\begin{align*}
\left[(\partial_{j}u_{\varepsilon})\,\eta^{2}\,\Phi((\vert Du_{\varepsilon}\vert-\lambda)_{+})\right]_{B_{p,\infty}^{2/p}(\mathbb{R}^{n})}^{p} & \leq\,c\int_{B_{r}}\boldsymbol{\lambda}(\vert Du_{\varepsilon}\vert)\,\Phi((\vert Du_{\varepsilon}\vert-\lambda)_{+})\,\vert D^{2}u_{\varepsilon}\vert^{2}\,\eta^{2}\,dx\\
 & \,\,\,\,\,\,\,+\,\frac{c}{r^{2}}\int_{B_{r}}\vert Du_{\varepsilon}\vert^{p}\,dx\,,
\end{align*}
where $c\equiv c(n,p)>0$. Combining the previous inequality,
(\ref{eq:extra}) and \eqref{eq:derf} with \eqref{eq:dualinequality}, and recalling the definition of $I_4$,
we get 
\begin{align}
 I_{4}\,&\le\,2\,\sum_{j=1}^n \left|\int_{B_{r}}(\partial_{j}f_{\varepsilon})(\partial_{j}u_{\varepsilon})\,\eta^{2}\,\Phi((\vert Du_{\varepsilon}\vert-\lambda)_{+})\,dx\right|\nonumber\\
&\le\,c\,\Vert f_{\varepsilon}\Vert_{B_{p',1}^{\frac{p-2}{p}}(B_{r})}\left[\left(\int_{B_{r}}\boldsymbol{\lambda}(\vert Du_{\varepsilon}\vert)\,\vert D^{2}u_{\varepsilon}\vert^{2}\,\eta^{2}\,\Phi((\vert Du_{\varepsilon}\vert-\lambda)_{+})\,dx\right)^{\frac{1}{p}}+\left(1+\,\frac{1}{r^{2/p}}\right)\Vert Du_{\varepsilon}\Vert_{L^{p}(B_{r})}\right],\label{eq:I4_A}
\end{align}
for a constant $c\equiv c(n,p)>0$.\\
\\
$\hspace*{1em}$Now, inserting estimates (\ref{I1_A}), 
 (\ref{I3_A}) and \eqref{eq:I4_A} in \eqref{eq:stimafond}, we obtain 
\begin{align*}
 & \int_{B_{r}}\boldsymbol{\lambda}(\vert Du_{\varepsilon}\vert)\,\vert D^{2}u_{\varepsilon}\vert^{2}\,\eta^{2}\,\Phi((\vert  Du_{\varepsilon}\vert-\lambda)_{+})\,dx\\
 & \,\,\,\,\,\,\,\leq\,c\left(\int_{B_{r}}\boldsymbol{\lambda}(\vert Du_{\varepsilon}\vert)\,\vert D^{2}u_{\varepsilon}\vert^{2}\,\eta^{2}\,\Phi((\vert  Du_{\varepsilon}\vert-\lambda)_{+})\,dx\right)^{\frac{1}{p}}\,\Vert f_{\varepsilon}\Vert_{B_{p',1}^{\frac{p-2}{p}}(B_{r})}\\
 & \,\,\,\,\,\,\,\,\,\,\,\,\,\,+\left(c+\,\frac{c}{r^{2/p}}\right)\Vert Du_{\varepsilon}\Vert_{L^{p}(B_{r})}\,\Vert f_{\varepsilon}\Vert_{B_{p',1}^{\frac{p-2}{p}}(B_{r})}\,+\,\frac{c}{r^{2}}\int_{B_{r}}(1+\vert Du_{\varepsilon}\vert^{2})^{\frac{p}{2}}\,dx\,,
\end{align*}
where $c\equiv c(n,p)>0$. Applying Young's inequality
to the first two terms on the right-hand side of the previous estimate, we get
\begin{align}\label{eq:key-ineq}
&\int_{B_{r}}\boldsymbol{\lambda}(\vert Du_{\varepsilon}\vert)\,\vert D^{2}u_{\varepsilon}\vert^{2}\,\eta^{2}\,\Phi((\vert  Du_{\varepsilon}\vert-\lambda)_{+})\,dx\nonumber\\
&\,\,\,\,\,\,\,\leq\left(C+\,\frac{C}{r^{2}}\right)\int_{B_{r}}(1+\vert Du_{\varepsilon}\vert^{2})^{\frac{p}{2}}\,dx\,+\,C\,\Vert f_{\varepsilon}\Vert_{B_{p',1}^{\frac{p-2}{p}}(B_{r})}^{p'},
\end{align}for some constant $C\equiv C(n,p)>0$.\\
$\hspace*{1em}$At
this point, recalling the definition of $\mathcal{V}_{\lambda}$
in (\ref{eq:Gfun})$-$(\ref{eq:Vfun}), a straightforward computation
reveals that, for every $j\in\{1,\ldots,n\}$, we have
\begin{align*}
\partial_{j}\mathcal{V}_{\lambda}(Du_{\varepsilon})\,=&\,\,\,\frac{(\vert Du_{\varepsilon}\vert-\lambda)_{+}^{\frac{p}{2}+\frac{1}{p-1}}}{\vert Du_{\varepsilon}\vert^{2}\,[\lambda+(\vert Du_{\varepsilon}\vert-\lambda)_{+}]^{1+\frac{1}{p-1}}}\,\langle Du_{\varepsilon},\partial_{j}Du_{\varepsilon}\rangle\,Du_{\varepsilon}\\
&\,\,+\,\mathcal{G}_{\lambda}((\vert Du_{\varepsilon}\vert-\lambda)_{+})\left[\frac{\partial_{j}Du_{\varepsilon}}{\vert Du_{\varepsilon}\vert}\,-\,\frac{\langle Du_{\varepsilon},\partial_{j}Du_{\varepsilon}\rangle}{\vert Du_{\varepsilon}\vert^{3}}\,Du_{\varepsilon}\right]
\end{align*}
if $\vert Du_{\varepsilon}\vert>\lambda$, and $\partial_{j}\mathcal{V}_{\lambda}(Du_{\varepsilon})=0$
otherwise. In the set $\{\vert Du_{\varepsilon}\vert>\lambda\}$, this yields
\begin{equation}
\vert D\mathcal{V}_{\lambda}(Du_{\varepsilon})\vert^{2}\leq\,\mathbf{B}_{1}+\mathbf{B}_{2}\,,\label{eq:B1+B2}
\end{equation}
where we define 
\[
\mathbf{B}_{1}:=\,2\,\frac{(\vert Du_{\varepsilon}\vert-\lambda)_{+}^{p+\frac{2}{p-1}}\,\vert D^{2}u_{\varepsilon}\vert^{2}}{[\lambda+(\vert Du_{\varepsilon}\vert-\lambda)_{+}]^{2+\frac{2}{p-1}}}
\]
and 
\[
\mathbf{B}_{2}:=\,8\,\frac{[\mathcal{G}_{\lambda}((\vert Du_{\varepsilon}\vert-\lambda)_{+})]^{2}\,\vert D^{2}u_{\varepsilon}\vert^{2}}{\vert Du_{\varepsilon}\vert^{2}}\,.
\]
We now estimate $\mathbf{B}_{1}$ and $\mathbf{B}_{2}$ separately
in the set where $\vert Du_{\varepsilon}\vert>\lambda$, since both
$\mathbf{B}_{1}$ and $\mathbf{B}_{2}$ vanish for $0<\vert Du_{\varepsilon}\vert\leq\lambda$.
Recalling the definitions \eqref{eq:minaut} and \eqref{eq:Phi}, we immediately have
\begin{align}
\mathbf{B}_{1}\,&=\,2\,\frac{(\vert Du_{\varepsilon}\vert-\lambda)^{p-1}\,\vert D^{2}u_{\varepsilon}\vert^{2}}{\vert Du_{\varepsilon}\vert}\left(\frac{\vert Du_{\varepsilon}\vert-\lambda}{\vert Du_{\varepsilon}\vert}\right)^{1+\frac{2}{p-1}}\nonumber\\
&=\,2\,\boldsymbol{\lambda}(\vert Du_{\varepsilon}\vert)\,\vert D^{2}u_{\varepsilon}\vert^{2}\,\Phi(\vert Du_{\varepsilon}\vert-\lambda)\,.\label{eq:B1}
\end{align}
As for $\mathbf{B}_{2}$, by Lemma \ref{lem:Glemma1} we obtain
\begin{align}\label{eq:B2}
\mathbf{B}_{2}\,&\leq\,\frac{32}{p^{2}}\,\,\frac{(\vert Du_{\varepsilon}\vert-\lambda)^{p+\frac{2p}{p-1}}\,\vert D^{2}u_{\varepsilon}\vert^{2}}{\vert Du_{\varepsilon}\vert^{2+\frac{2p}{p-1}}}\nonumber\\
&=\,\frac{32}{p^{2}}\,\,\frac{(\vert Du_{\varepsilon}\vert-\lambda)^{p-1}\,\vert D^{2}u_{\varepsilon}\vert^{2}}{\vert Du_{\varepsilon}\vert}\left(\frac{\vert Du_{\varepsilon}\vert-\lambda}{\vert Du_{\varepsilon}\vert}\right)^{1+\frac{2p}{p-1}}\nonumber\\
&=\,\frac{32}{p^{2}}\,\boldsymbol{\lambda}(\vert Du_{\varepsilon}\vert)\,\vert D^{2}u_{\varepsilon}\vert^{2}\,\Phi(\vert Du_{\varepsilon}\vert-\lambda)\,.
\end{align}
Joining estimates (\ref{eq:B1+B2})$-$\eqref{eq:B2}, we then find
\begin{equation}
\int_{B_{r}}\vert D\mathcal{V}_{\lambda}(Du_{\varepsilon})\vert^{2}\,\eta^{2}\,dx\,\leq\,c(p)\int_{B_{r}}\boldsymbol{\lambda}(\vert Du_{\varepsilon}\vert)\,\vert D^{2}u_{\varepsilon}\vert^{2}\,\eta^{2}\,\Phi((\vert Du_{\varepsilon}\vert-\lambda)_{+})\,dx\,,\label{eq:combo}
\end{equation}
which combined with \eqref{eq:key-ineq}, gives
\[
\int_{B_{r}}\vert D\mathcal{V}_{\lambda}(Du_{\varepsilon})\vert^{2}\,\eta^{2}\,dx\,\leq\left(C+\,\frac{C}{r^{2}}\right)\int_{B_{r}}(1+\vert Du_{\varepsilon}\vert^{2})^{\frac{p}{2}}\,dx\,+\,C\,\Vert f_{\varepsilon}\Vert_{B_{p',1}^{\frac{p-2}{p}}(B_{r})}^{p'}.
\]
$\hspace*{1em}$Let us now consider the same $\varepsilon_{0}\in(0,1]$
as in Proposition \ref{prop:uniform} and let $\varepsilon\in(0,\varepsilon_{0}]$.
Then, recalling that $\eta\equiv1$ on $\overline{B}_{r/2}$ and applying
estimate (\ref{eq:unifenest}), we obtain
\[
\int_{B_{r/2}}\vert D\mathcal{V}_{\lambda}(Du_{\varepsilon})\vert^{2}\,dx\,\leq\left(C+\,\frac{C}{r^{2}}\right)\left[1+\lambda^{p}+\,\Vert Du\Vert_{L^{p}(B_{R})}^{p}\,+\Vert f\Vert_{L^{(p^{*})'}(B_{R})}^{p'}\right]+\,C\,\Vert f_{\varepsilon}\Vert_{B_{p',1}^{\frac{p-2}{p}}(B_{r})}^{p'},
\]
where we have used the fact that $r<R\leq1$. Since $(p^{*})'<p'$,
using Hölder's inequality, from the above estimate we get 
\[
\int_{B_{r/2}}\vert D\mathcal{V}_{\lambda}(Du_{\varepsilon})\vert^{2}\,dx\,\leq\left(C+\,\frac{C}{r^{2}}\right)\left[1+\lambda^{p}+\,\Vert Du\Vert_{L^{p}(B_{R})}^{p}\,+\Vert f\Vert_{L^{p'}(B_{R})}^{p'}\right]+\,C\,\Vert f_{\varepsilon}\Vert_{B_{p',1}^{\frac{p-2}{p}}(B_{r})}^{p'}.
\]
Furthermore, there exists a positive number $\varepsilon_{1}\leq\varepsilon_{0}$
such that 
\[
\Vert f_{\varepsilon}\Vert_{B_{p',1}^{\frac{p-2}{p}}(B_{r})}\,\leq\,\Vert f\Vert_{B_{p',1}^{\frac{p-2}{p}}(B_{R})}<+\infty\,\,\,\,\,\,\,\,\,\,\mathrm{for\,\,every\,\,}\varepsilon\in(0,\varepsilon_{1}].
\]
\foreignlanguage{british}{Combining the last two estimates for $\varepsilon\in(0,\varepsilon_{1}]$,
we conclude the proof} in the case $\lambda>0$.\\
$\hspace*{1em}$Finally, when $\lambda=0$ the above proof can be
greatly simplified, as we can choose $\Phi\equiv1$ and we have 
\[
\mathcal{V}_{0}(Du_{\varepsilon})=\,\frac{2}{p}\,H_{\frac{p}{2}}(Du_{\varepsilon})=\,\frac{2}{p}\,\vert Du_{\varepsilon}\vert^{\frac{p-2}{2}}Du_{\varepsilon}\,.
\]
In this regard, we leave the details to the reader.
\end{proof}
\noindent $\hspace*{1em}$Combining Lemma \ref{lem:Giusti1} with
estimate \eqref{eq:Sobolev1}, we obtain the following
\begin{cor}
\noindent Let $\varepsilon_{1}\in(0,1]$ be the constant from Proposition
\ref{prop:uniform2}. Under the assumptions of Theorem \ref{thm:theo1}
and with the notation above, for every pair of concentric balls $B_{r/4}\subset B_{r}\subset B_{R}$
we have\begin{align}\label{eq:Sobolev2}
&\int_{B_{r/4}}\left|\tau_{j,h}\mathcal{V}_{\lambda}(Du_{\varepsilon})\right|^{2}dx\nonumber\\
&\,\,\,\,\,\,\,\leq\,\frac{C\,|h|^{2}\,(1+r^{2})}{r^{2}}\left[1+\lambda^{p}+\,\Vert Du\Vert_{L^{p}(B_{R})}^{p}\,+\Vert f\Vert_{L^{p'}(B_{R})}^{p'}\right]+\,C\,\vert h\vert^{2}\,\Vert f\Vert_{B_{p',1}^{\frac{p-2}{p}}(B_{R})}^{p'}\,,
\end{align} for every $j\in\{1,\ldots,n\}$, for every $h\in\mathbb{R}$ such
that $\vert h\vert<\frac{r}{8}$, for every $\varepsilon\in(0,\varepsilon_{1}]$
and a positive constant $C$ depending only
on $n$ and $p$.
\end{cor}

\noindent $\hspace*{1em}$We are now in a position to prove Theorem
\ref{thm:theo1}.
\begin{proof}[\textbf{{Proof of Theorem~\ref{thm:theo1}}}]
 Consider the same $\varepsilon_{1}\in(0,1]$ as in Proposition \ref{prop:uniform2}
and let $\varepsilon\in(0,\varepsilon_{1}]$. Moreover, let $u_{\varepsilon}$
be the unique energy solution of the Dirichlet problem (\ref{eq:approximation}).
Now we fix three concentric balls $B_{r/4}$, $B_{r/2}$ and $B_{r}$,
with $B_{r}\subset B_{R}\Subset\Omega$, $R\le1$, and use the finite
difference operator $\tau_{j,h}$ defined in Section \ref{subsec:DiffOpe},
for increments $h\in\mathbb{R}\setminus\{0\}$ such that $\vert h\vert<\frac{r}{8}$.
In what follows, we will denote by $C$ a positive constant which
neither depends on $\varepsilon$ nor on $h$. In order to obtain
an estimate for the finite difference $\tau_{j,h}\mathcal{V}_{\lambda}(Du)$,
we use the following comparison argument: 
\begin{align*}
 & \int_{B_{r/4}}\left|\tau_{j,h}\mathcal{V}_{\lambda}(Du)\right|^{2}dx\\
 & \,\,\,\,\,\,\,\leq\,4\int_{B_{r/4}}\left|\tau_{j,h}\mathcal{V}_{\lambda}(Du_{\varepsilon})\right|^{2}dx\,+\,4\int_{B_{r/4}}\left|\mathcal{V}_{\lambda}(Du_{\varepsilon})-\mathcal{V}_{\lambda}(Du)\right|^{2}dx\\
 & \,\,\,\,\,\,\,\,\,\,\,\,\,\,+\,4\int_{B_{r/4}}\left|\mathcal{V}_{\lambda}(Du_{\varepsilon}(x+he_{j}))-\mathcal{V}_{\lambda}(Du(x+he_{j}))\right|^{2}dx\\
 & \,\,\,\,\,\,\,\leq\,4\int_{B_{r/4}}\left|\tau_{j,h}\mathcal{V}_{\lambda}(Du_{\varepsilon})\right|^{2}dx\,+\,8\int_{B_{R}}\left|\mathcal{V}_{\lambda}(Du_{\varepsilon})-\mathcal{V}_{\lambda}(Du)\right|^{2}dx,
\end{align*}
where we have used the second statement in Lemma \ref{lem:Giusti1}.
Combining the previous estimate with \eqref{eq:Sobolev2} and (\ref{eq:comparison}),
for every $j\in\{1,\ldots,n\}$ we get\begin{align}\label{eq:teo1new}
&\int_{B_{r/4}}\left|\tau_{j,h}\mathcal{V}_{\lambda}(Du)\right|^{2}dx\nonumber\\
&\,\,\,\,\,\,\,\leq\,\frac{C\,|h|^{2}\,(1+r^{2})}{r^{2}}\left[1+\lambda^{p}+\,\Vert Du\Vert_{L^{p}(B_{R})}^{p}\,+\Vert f\Vert_{L^{p'}(B_{R})}^{p'}\right]+\,C\,\vert h\vert^{2}\,\Vert f\Vert_{B_{p',1}^{\frac{p-2}{p}}(B_{R})}^{p'}\nonumber\\
&\,\,\,\,\,\,\,\,\,\,\,\,\,\,+\,C\,\Vert f_{\varepsilon}-f\Vert_{L^{(p^{*})'}(B_{R})}\left[1+\lambda+\,\Vert Du\Vert_{L^{p}(B_{R})}\,+\,\Vert f\Vert_{L^{(p^{*})'}(B_{R})}^{\frac{1}{p-1}}\right]\nonumber\\
&\,\,\,\,\,\,\,\,\,\,\,\,\,\,+\,C\,\varepsilon\left[1+\lambda^{p}\,+\,\Vert Du\Vert_{L^{p}(B_{R})}^{p}\,+\,\Vert f\Vert_{L^{(p^{*})'}(B_{R})}^{p'}\right],
\end{align}which holds for every sufficiently small $h\in\mathbb{R}\setminus\{0\}$
and a constant $C\equiv C(n,p)>0$. Therefore, recalling \eqref{eq:strongconv} and letting
$\varepsilon\searrow0$ in \eqref{eq:teo1new}, we obtain
\begin{align*}
 & \int_{B_{r/4}}\left|\Delta_{j,h}\mathcal{V}_{\lambda}(Du)\right|^{2}dx\\
 & \,\,\,\,\,\,\,\leq\left(C+\,\frac{C}{r^{2}}\right)\left[1+\lambda^{p}+\,\Vert Du\Vert_{L^{p}(B_{R})}^{p}\,+\Vert f\Vert_{L^{p'}(B_{R})}^{p'}\right]+\,C\,\Vert f\Vert_{B_{p',1}^{\frac{p-2}{p}}(B_{R})}^{p'}.
\end{align*}
Since the above inequality holds for every $j\in\{1,\ldots,n\}$ and
every sufficiently small $h\neq0$, by Lemma \ref{lem:RappIncre}
we may conclude that $\mathcal{V}_{\lambda}(Du)\in W_{loc}^{1,2}(\Omega,\mathbb{R}^{n})$.
Moreover, letting $h\rightarrow0$ in the previous inequality, we
obtain estimate \eqref{eq:estteo1} for every ball $B_{R}\Subset\Omega$
with $R\le1$. The validity of \eqref{eq:estteo1} for arbitrary balls
follows from a standard covering argument.
\end{proof}

\section{Proof of Theorem \ref{thm:theo2} \label{sec:teorema2}}

\noindent $\hspace*{1em}$This section is devoted to the proof of
Theorem \ref{thm:theo2}. Actually, here we limit ourselves to deriving
the \textit{a priori} estimates for $1<p\leq2$, since inequality
\eqref{eq:estteo2} can be obtained using the same arguments presented
in Section \ref{sec:theo1}. In what follows, we shall keep the notations
used for the proof of Proposition \ref{prop:uniform2}.
\begin{proof}[\textbf{{Proof of Theorem~\ref{thm:theo2}}}]
\noindent  Let us first assume that $\lambda>0$. Arguing as in the
proof of Proposition \ref{prop:uniform2}, we define the integrals
$I_{1}$$-$$I_{4}$ exactly as in \eqref{eq:integrals}. We need
to treat differently only the integrals $I_{2}$ and $I_{4}$, in
which the new assumptions $1<p\leq2$ and $f\in L_{loc}^{\frac{np}{n(p-1)+2-p}}(\Omega)$
are involved. Under these new hypotheses and for $\vert Du_\varepsilon \vert > \lambda$, equality \eqref{eq:equa3} is replaced by
\begin{align}
&\sum_{j=1}^{n}\,\langle D^{2}G_{\varepsilon}(Du_{\varepsilon})\,D(\partial_{j}u_{\varepsilon}),D[(\vert Du_\varepsilon\vert-\lambda)_{+}]\rangle\,(\partial_{j}u_{\varepsilon})\nonumber\\
& \,\,\,\,\,\,\,=\left[(p-1)\,\frac{\boldsymbol{\Lambda}(\vert Du_{\varepsilon}\vert)}{\vert Du_{\varepsilon}\vert}\,-\,\frac{\boldsymbol{\lambda}(\vert Du_{\varepsilon}\vert)}{(p-1)\vert Du_{\varepsilon}\vert}\,+\,\varepsilon\,(p-2)\,(1+\vert Du_{\varepsilon}\vert^{2})^{\frac{p-4}{2}}\,\vert Du_{\varepsilon}\vert\right]\nonumber\\
 & \,\,\,\,\,\,\,\,\,\,\,\,\,\,\cdot\left[\sum_{k=1}^{n}\,(\partial_{k}u_{\varepsilon})\,\partial_{k}[(\vert Du_\varepsilon\vert-\lambda)_{+}]\right]^{2}\nonumber\\
 & \,\,\,\,\,\,\,\,\,\,\,\,\,\,+\left[\frac{\boldsymbol{\lambda}(\vert Du_\varepsilon\vert)\,\vert Du_\varepsilon\vert}{p-1}\,+\,\varepsilon\,(1+\vert Du_{\varepsilon}\vert^{2})^{\frac{p-2}{2}}\,\vert Du_{\varepsilon}\vert\right]\,\big\vert D[(\vert Du_\varepsilon\vert-\lambda)_{+}]\big\vert^{2},
\label{eq:equa3subq}
\end{align}
where we have used the definitions \eqref{eq:minaut} and \eqref{eq:maxaut} again. It comes out that $I_{2}$ is non-negative, as in the
super-quadratic case. Indeed, estimates \eqref{eq:equa3subq} and \eqref{eq:dersec}
lead us to 
\begin{align*}
I_{2}\, & \ge\,2\int_{B_{r}}\eta^{2}\,\Phi'((\vert Du_\varepsilon\vert-\lambda)_{+})\\
 & \,\,\,\,\,\,\,\,\,\,\cdot\Bigg\{\left[(p-1)\,\frac{\boldsymbol{\Lambda}(\vert Du_{\varepsilon}\vert)}{\vert Du_{\varepsilon}\vert}\,-\,\frac{\boldsymbol{\lambda}(\vert Du_{\varepsilon}\vert)}{(p-1)\vert Du_{\varepsilon}\vert}+\varepsilon(p-2)(1+|Du_{\varepsilon}|^{2})^{\frac{p-4}{2}}|Du_{\varepsilon}|\right]\\
 & \,\,\,\,\,\,\,\,\,\,\,\,\,\,\,\cdot\left[\sum_{k=1}^{n}\,(\partial_{k}u_{\varepsilon})\,\partial_{k}[(\vert Du_\varepsilon\vert-\lambda)_{+}]\right]^{2}+\,\frac{\boldsymbol{\lambda}(\vert Du_\varepsilon\vert)\,\vert Du_\varepsilon\vert}{p-1}\,\big\vert D[(\vert Du_\varepsilon\vert-\lambda)_{+}]\big\vert^{2}\\
 & \,\,\,\,\,\,\,\,\,\,\,\,\,\,\,+\,\varepsilon\,(1+|Du_{\varepsilon}|^{2})^{\frac{p-4}{2}}\,|Du_{\varepsilon}|^{3}\,\big\vert D[(\vert Du_\varepsilon\vert-\lambda)_{+}]\big\vert^{2}\Bigg\}\,dx\\
 & \ge\,2\int_{B_{r}}\eta^{2}\,\Phi'((\vert Du_\varepsilon\vert-\lambda)_{+})\left[(p-1)\,\frac{\boldsymbol{\Lambda}(\vert Du_{\varepsilon}\vert)}{\vert Du_{\varepsilon}\vert}\,+\,\varepsilon(p-1)(1+|Du_{\varepsilon}|^{2})^{\frac{p-4}{2}}|Du_{\varepsilon}|\right]\\
 & \,\,\,\,\,\,\,\,\,\,\,\,\,\,\,\cdot\left[\sum_{k=1}^{n}\,(\partial_{k}u_{\varepsilon})\,\partial_{k}[(\vert Du_\varepsilon\vert-\lambda)_{+}]\right]^{2}dx\,\ge\,0\,,
\end{align*}
where the function $\Phi$ is defined in \eqref{eq:Phi}. Then, using (\ref{eq:eta}), (\ref{eq:Phi}), Lemma \ref{lem:qform} and the fact that $I_{2}\ge0$,
from \eqref{eq:integrals} we now obtain 
\begin{align}
 & \int_{B_{r}}\boldsymbol{\lambda}(\vert Du_{\varepsilon}\vert)\,\vert D^{2}u_{\varepsilon}\vert^{2}\,\eta^{2}\,\Phi((\vert Du_\varepsilon\vert-\lambda)_{+})\,dx\nonumber \\
 & \,\,\,\,\,\,\,\,\,\leq\,\frac{c(p)}{r^{2}}\int_{B_{r}}(1+\vert Du_{\varepsilon}\vert^{2})^{\frac{p}{2}}\,dx\,+\,{\frac{c(p)}{r^{2}}\int_{B_{r}}\boldsymbol{\Lambda}(\vert Du_{\varepsilon}\vert)\,\vert Du_{\varepsilon}\vert^{2}\,\Phi((\vert Du_\varepsilon\vert-\lambda)_{+})\,dx}\nonumber \\
 & \,\,\,\,\,\,\,\,\,\,\,\,\,\,\,\,+\,c(p)\,\sum_{j=1}^{n}\int_{B_{r}}(\partial_{j}f_{\varepsilon})(\partial_{j}u_{\varepsilon})\,\eta^{2}\,\Phi((\vert Du_\varepsilon\vert-\lambda)_{+})\,dx\nonumber \\
 & \,\,\,\,\,\,\,\,\,\leq\,\frac{c(p)}{r^{2}}\int_{B_{r}}(1+\vert Du_{\varepsilon}\vert^{2})^{\frac{p}{2}}\,dx\,+c(p)\,\sum_{j=1}^{n}\int_{B_{r}}(\partial_{j}f_{\varepsilon})(\partial_{j}u_{\varepsilon})\,\eta^{2}\,\Phi((\vert Du_\varepsilon\vert-\lambda)_{+})\,dx\,.\label{eq:equa5.1}
\end{align}
At this
point, we integrate by parts and then apply Hölder's inequality in
the second integral on right-hand side of \eqref{eq:equa5.1}. This
gives 
\begin{equation}
\begin{split} & \left|\int_{B_{r}}(\partial_{j}f_{\varepsilon})(\partial_{j}u_{\varepsilon})\,\eta^{2}\,\Phi((\vert Du_\varepsilon\vert-\lambda)_{+})\,dx\right|\\
 & \,\,\,\,\,\,\,\leq\,\Vert f_{\varepsilon}\Vert_{L^{\frac{np}{n(p-1)+2-p}}(B_{r})}\,\Vert\partial_{j}[(\partial_{j}u_{\varepsilon})\,\eta^{2}\,\Phi((\vert Du_\varepsilon\vert-\lambda)_{+})]\Vert_{L^{\frac{np}{n-2+p}}(\mathbb{R}^{n})}.
\end{split}
\label{eq:dualinequality2}
\end{equation}
From now on, we will only deal with the case $n\geq3$ and $1<p<2$,
since the remaining cases imply that 
\[
\frac{np}{n(p-1)+2-p}\,=\,\frac{np}{n-2+p}\,=\,2
\]
and can be addressed by suitably modifying the arguments used in Section
\ref{sec:theo1}. Note that in the case $p=n=2$, we have $(p^{*})'<p'=2$,
and therefore we can continue to argue as in the proof of Proposition
\ref{prop:uniform2}.\\
For ease of notation, we now set 
\[
Z(x):=\,\eta^{2}(x)\cdot(\partial_{j}u_{\varepsilon}(x))\cdot\Phi((\vert Du_\varepsilon\vert-\lambda)_{+})\,.
\]
Since $\frac{np}{n-2+p}<2$ {for $n\geq3$ and $1<p<2$}, an application
of Hölder's inequality and Lemma \ref{D1} yield 
\begin{align*}
 & \int_{\mathbb{R}^{n}}\left|\tau_{j,h}Z(x)\right|^{\frac{np}{n-2+p}}\,dx\\
 & =\,\int_{\mathbb{R}^{n}}\left|\tau_{j,h}Z(x)\right|^{\frac{np}{n-2+p}}\left(\vert Z(x+he_{j})\vert^{2}\,+\,\vert Z(x)\vert^{2}\right)^{\frac{(p-2)}{4}\frac{np}{n-2+p}}\left(\vert Z(x+he_{j})\vert^{2}\,+\,\vert Z(x)\vert^{2}\right)^{\frac{(2-p)}{4}\frac{np}{n-2+p}}\,dx\\
 & \le\,c(n,p)\left(\int_{\mathbb{R}^{n}}\left|\tau_{j,h}Z(x)\right|^{2}\left(\vert Z(x+he_{j})\vert^{2}\,+\,\vert Z(x)\vert^{2}\right)^{\frac{p-2}{2}}\,dx\right)^{\frac{np}{2(n-2+p)}}\left(\int_{\mathbb{R}^{n}}\vert Z(x)\vert^{\frac{np}{n-2}}\,dx\right)^{\frac{(n-2)(2-p)}{2(n-2+p)}}\\
 & \le\,c(n,p)\left(\int_{\mathbb{R}^{n}}\left|\tau_{j,h}\left(\vert Z(x)\vert^{\frac{p-2}{2}}Z(x)\right)\right|^{2}\,dx\right)^{\frac{np}{2(n-2+p)}}\left(\int_{\mathbb{R}^{n}}\vert Z(x)\vert^{\frac{np}{n-2}}\,dx\right)^{\frac{(n-2)(2-p)}{2(n-2+p)}}
\end{align*}
for every $h\in\mathbb{R}\setminus\{0\}$. Dividing both sides by
$\vert h\vert^{\frac{np}{n-2+p}}$ and letting $h\rightarrow0$, by
virtue of Lemma \ref{lem:RappIncre} we obtain 
\begin{align}
 & \int_{\mathbb{R}^{n}}\left|\partial_{j}Z(x)\right|^{\frac{np}{n-2+p}}\,dx\nonumber \\
 & \le c(n,p)\left(\int_{\mathbb{R}^{n}}\left|\partial_{j}\left(\vert Z(x)\vert^{\frac{p-2}{2}}Z(x)\right)\right|^{2}\,dx\right)^{\frac{np}{2(n-2+p)}}\left(\int_{\mathbb{R}^{n}}\vert Z(x)\vert^{\frac{np}{n-2}}\,dx\right)^{\frac{(n-2)(2-p)}{2(n-2+p)}}\nonumber \\
 & \le c(n,p)\left(\int_{\mathbb{R}^{n}}\left|D\left(\vert Z(x)\vert^{\frac{p-2}{2}}Z(x)\right)\right|^{2}\,dx\right)^{\frac{np}{2(n-2+p)}}\left(\int_{\mathbb{R}^{n}}\vert Z(x)\vert^{\frac{np}{n-2}}\,dx\right)^{\frac{(n-2)(2-p)}{2(n-2+p)}}\nonumber \\
 & =c(n,p)\left(\int_{\mathbb{R}^{n}}\left|D\left(\vert Z(x)\vert^{\frac{p-2}{2}}Z(x)\right)\right|^{2}\,dx\right)^{\frac{np}{2(n-2+p)}}\left(\int_{\mathbb{R}^{n}}\Big\vert|Z(x)|^{\frac{p-2}{2}}Z(x)\Big\vert^{\frac{2n}{n-2}}\,dx\right)^{\frac{(n-2)(2-p)}{2(n-2+p)}}\nonumber \\
 & \le c(n,p)\left(\int_{\mathbb{R}^{n}}\left|D\left(\vert Z(x)\vert^{\frac{p-2}{2}}Z(x)\right)\right|^{2}\,dx\right)^{\frac{np}{2(n-2+p)}}\left(\int_{\mathbb{R}^{n}}\left|D\left(\vert Z(x)\vert^{\frac{p-2}{2}}Z(x)\right)\right|^{2}\,dx\right)^{\frac{n(2-p)}{2(n-2+p)}}\nonumber \\
 & =c(n,p)\left(\int_{\mathbb{R}^{n}}\left|D\left(\vert Z(x)\vert^{\frac{p-2}{2}}Z(x)\right)\right|^{2}\,dx\right)^{\frac{n}{n-2+p}}.\label{eq:modifica02}
\end{align}
Recalling the definition of $Z$, calculating the gradient in the
right-hand side of \eqref{eq:modifica02}, using the properties of
$\eta$ and recalling that $\Phi\le1$, we get 
\begin{align}
 & \Vert\,\partial_{j}[\eta^{2}(\partial_{j}u_{\varepsilon})\,\Phi((\vert Du_\varepsilon\vert-\lambda)_{+})]\Vert_{L^{\frac{np}{n-2+p}}(\mathbb{R}^{n})}\nonumber \\
 & \,\,\,\,\,\,\,\le\,c\left(\int_{\mathbb{R}^{n}}\left|D\left(\eta^{p}\vert\partial_{j}u_{\varepsilon}\vert^{\frac{p-2}{2}}(\partial_{j}u_{\varepsilon})\,[\Phi((\vert Du_\varepsilon\vert-\lambda)_{+})]^{\frac{p}{2}}\right)\right|^{2}\,dx\right)^{\frac{1}{p}}\nonumber \\
 & \,\,\,\,\,\,\,\le\,c\left(\int_{B_{r}}\eta^{2p}\left|D\left(\vert\partial_{j}u_{\varepsilon}\vert^{\frac{p-2}{2}}(\partial_{j}u_{\varepsilon})\,[\Phi((\vert Du_\varepsilon\vert-\lambda)_{+})]^{\frac{p}{2}}\right)\right|^{2}\,dx\right)^{\frac{1}{p}}\nonumber \\
 & \,\,\,\,\,\,\,\,\,\,\,\,\,\,+c\left(\int_{B_{r}}\eta^{2p-2}\left|D\eta\right|^{2}\vert Du_{\varepsilon}\vert^{p}\,dx\right)^{\frac{1}{p}}.\label{eq:equa2.2}
\end{align}
Inserting \eqref{eq:equa2.2} into \eqref{eq:dualinequality2} and
using the properties of $\eta$, we obtain
\begin{align}\label{eq:teo2new}
&\left|\int_{B_{r}}(\partial_{j}f_{\varepsilon})(\partial_{j}u_{\varepsilon})\,\eta^{2}\,\Phi((\vert Du_\varepsilon\vert-\lambda)_{+})\,dx\right|\nonumber\\
&\,\,\,\,\,\,\,\leq\,c\,\Vert f_{\varepsilon}\Vert_{L^{\frac{np}{n(p-1)+2-p}}(B_{r})}\,\left(\int_{B_{r}}\eta^{2}\left|D\left(\vert\partial_{j}u_{\varepsilon}\vert^{\frac{p-2}{2}}(\partial_{j}u_{\varepsilon})\,[\Phi((\vert Du_\varepsilon\vert-\lambda)_{+})]^{\frac{p}{2}}\right)\right|^{2}\,dx\right)^{\frac{1}{p}}\nonumber\\
&\,\,\,\,\,\,\,\,\,\,\,\,\,\,+\,\frac{c}{r^{2/p}}\,\Vert f_{\varepsilon}\Vert_{L^{\frac{np}{n(p-1)+2-p}}(B_{r})}\,\left(\int_{B_{r}}\vert Du_{\varepsilon}\vert^{p}\,dx\right)^{\frac{1}{p}}.
\end{align}
Now, combining \eqref{eq:teo2new} with \eqref{eq:equa5.1}, we have
\begin{align*}
 & \int_{B_{r}}\boldsymbol{\lambda}(\vert Du_{\varepsilon}\vert)\,\vert D^{2}u_{\varepsilon}\vert^{2}\,\eta^{2}\,\Phi((\vert Du_\varepsilon\vert-\lambda)_{+})\,dx\\
 & \,\,\,\,\,\,\,\leq\,\frac{c}{r^{2}}\int_{B_{r}}(1+\vert Du_{\varepsilon}\vert^{2})^{\frac{p}{2}}\,dx\,+\,\frac{c}{r^{2/p}}\,\Vert f_{\varepsilon}\Vert_{L^{\frac{np}{n(p-1)+2-p}}(B_{r})}\,\Vert Du_{\varepsilon}\Vert_{L^{p}(B_{r})}\\
 & \,\,\,\,\,\,\,\,\,\,\,\,\,\,\,+\,c\,\Vert f_{\varepsilon}\Vert_{L^{\frac{np}{n(p-1)+2-p}}(B_{r})}\,\sum_{j=1}^{n}\left(\int_{B_{r}}\eta^{2}\left|D\left(\vert\partial_{j}u_{\varepsilon}\vert^{\frac{p-2}{2}}(\partial_{j}u_{\varepsilon})\,[\Phi((\vert Du_\varepsilon\vert-\lambda)_{+})]^{\frac{p}{2}}\right)\right|^{2}\,dx\right)^{\frac{1}{p}}.
\end{align*}
The last integral can be estimated
using \eqref{eq:A1+A2}, \eqref{est:A1} and \eqref{est:A2}. Thus we infer 
\begin{align*}
 & \int_{B_{r}}\boldsymbol{\lambda}(\vert Du_{\varepsilon}\vert)\,\vert D^{2}u_{\varepsilon}\vert^{2}\,\eta^{2}\,\Phi((\vert Du_\varepsilon\vert-\lambda)_{+})\,dx\\
 & \,\,\,\,\,\,\,\leq\,\frac{c}{r^{2}}\int_{B_{r}}(1+\vert Du_{\varepsilon}\vert^{2})^{\frac{p}{2}}\,dx\,+\,\frac{c}{r^{2/p}}\,\Vert f_{\varepsilon}\Vert_{L^{\frac{np}{n(p-1)+2-p}}(B_{r})}\,\Vert Du_{\varepsilon}\Vert_{L^{p}(B_{r})}\\
 & \,\,\,\,\,\,\,\,\,\,\,\,\,\,\,+\,c\,\Vert f_{\varepsilon}\Vert_{L^{\frac{np}{n(p-1)+2-p}}(B_{r})}\left(\int_{B_{r}}\boldsymbol{\lambda}(\vert Du_{\varepsilon}\vert)\,\vert D^{2}u_{\varepsilon}\vert^{2}\,\eta^{2}\,\Phi((\vert Du_\varepsilon\vert-\lambda)_{+})\,dx\right)^{\frac{1}{p}},
\end{align*}
where $c\equiv c(n,p)>0$. Applying Young's inequality to reabsorb the last integral by the left-hand
side, and then using inequality (\ref{eq:combo}), we derive 
\begin{align*}
\int_{B_{r}}\vert D\mathcal{V}_{\lambda}(Du_{\varepsilon})\vert^{2}\,\eta^{2}\,dx\,\leq\,\frac{c}{r^{2}}\int_{B_{r}}(1+\vert Du_{\varepsilon}\vert^{2})^{\frac{p}{2}}\,dx\,+\,c\,\Vert f_{\varepsilon}\Vert_{L^{\frac{np}{n(p-1)+2-p}}(B_{r})}^{p'}\,.
\end{align*}
The desired conclusion follows by arguing as in the proofs of Proposition
\ref{prop:uniform2} and Theorem \ref{thm:theo1}, observing that
\[
(p^{*})'=\frac{np}{np-n+p}\,<\,\frac{np}{n(p-1)+2-p}\qquad\text{for every}\,\,p>1.
\]
$\hspace*{1em}$Finally, when $\lambda=0$ the above proof can be
greatly simplified, as we can choose $\Phi\equiv1$ and we have
\[
\mathcal{V}_{0}(Du_{\varepsilon})=\,\frac{2}{p}\,\vert Du_{\varepsilon}\vert^{\frac{p-2}{2}}Du_{\varepsilon}\,.
\]
We leave the details to the reader.
\end{proof}
\medskip{}

\noindent \textbf{Acknowledgements.} We would like to thank the reviewer for his/her valuable comments, which helped to improve our paper.\\
The authors are members of the Gruppo Nazionale per l’Analisi
Matematica, la Probabilità e le loro Applicazioni (GNAMPA) of the
Istituto Nazionale di Alta Matematica (INdAM). P. Ambrosio has been
partially supported through the INdAM$-$GNAMPA 2025 Project “Regolarità ed esistenza per operatori anisotropi” (CUP E5324001950001). In addition, P. Ambrosio acknowledges
financial support under the National Recovery and Resilience Plan (NRRP), Mission 4, Component 2, Investment 1.1, Call for tender No. 104 published on 02/02/2022 by the Italian Ministry of University and Research (MUR), funded by the European Union - NextGenerationEU -
Project PRIN\_CITTI 2022 - Title ``Regularity problems in sub-Riemannian
structures'' - CUP J53D23003760006 - Bando 2022 - Prot. 2022F4F2LH. A.G. Grimaldi and A. Passarelli
di Napoli have been partially supported through the INdAM$-$GNAMPA
2025 Project “Regolarità di soluzioni di equazioni paraboliche a crescita nonstandard degeneri” (CUP E5324001950001). A. Passarelli di Napoli has also
been supported 
by the Centro Nazionale per la Mobilità Sostenibile (CN00000023) -
Spoke 10 Logistica Merci (CUP E63C22000930007). A.G. Grimaldi has also been partially supported by PNRR - Missione 4 “Istruzione e Ricerca” Componente 2 “Dalla Ricerca all'Impresa” -
Investimento 1.2 “Finanziamento di progetti presentati da giovani ricercatori” (CUP I63C25000150004) and through the project “Geometric-Analytic Methods for PDEs and Applications (GAMPA)” - funded by European Union - NextGenerationEU - within the PRIN 2022 program (D.D. 104 - 02/02/2022 Ministero dell’Università e della Ricerca). This manuscript reflects only the authors’ views and opinions and the Ministry cannot be considered responsible for them.\medskip{}


\lyxaddress{\textbf{$\quad$}\\
$\hspace*{1em}$\textbf{Pasquale Ambrosio}\\
Dipartimento di Matematica, Università di Bologna\\
Piazza di Porta S. Donato 5, 40126 Bologna, Italy.\\
\textit{E-mail address}: pasquale.ambrosio@unibo.it}

\lyxaddress{$\hspace*{1em}$\textbf{Antonio Giuseppe Grimaldi}\\
Dipartimento di Ingegneria, Università degli Studi di Napoli “Parthenope''\\
Centro Direzionale Isola C4 - 80143 Napoli, Italy.\\
\textit{E-mail address}: antoniogiuseppe.grimaldi@collaboratore.uniparthenope.it}

\lyxaddress{$\hspace*{1em}$\textbf{Antonia Passarelli di Napoli}\\
Dipartimento di Matematica e Applicazioni ``R. Caccioppoli''\\
Università degli Studi di Napoli ``Federico II''\\
Via Cintia, 80126 Napoli, Italy.\\
\textit{E-mail address}: antpassa@unina.it}
\selectlanguage{british}%

\lyxaddress{\textbf{$\quad$}}
\end{document}